\documentclass[a4paper,12pt,reqno]{amsart}
\usepackage{amsmath, amsfonts, amssymb, mathtools}
\usepackage[hang,flushmargin]{footmisc} 
\usepackage{amsthm}
\usepackage{comment}
\usepackage{hyperref}
\usepackage{etoolbox}
\usepackage[dvistyle, colorinlistoftodos]{todonotes}

\newtheorem{theorem}{Theorem}[section]
\newtheorem{lemma}[theorem]{Lemma}

\newtheorem{proposition}[theorem]{Proposition}
\newtheorem{question}[theorem]{Question}
\newtheorem{problem}[theorem]{Problem}

\theoremstyle{definition}
\newtheorem{defn}[theorem]{Definition}
\newtheorem{remark}[theorem]{Remark}
\newtheorem{example}[theorem]{Example}

\newtheoremstyle{break}
  {}
  {}
  {\itshape}
  {}
  {\bfseries}
  {.}
  {\newline}
  {}

\theoremstyle{break}

\def\E{\mathbb{E}}
\def\Z{\mathbb{Z}}
\def\R{\mathbb{R}}
\def\T{\mathbb{T}}
\def\C{\mathbb{C}}
\def\N{\mathbb{N}}

\def\cD{\mathcal{D}}
\def\cL{\mathcal{L}}

\def\X{\mathcal{X}}

\newcommand{\ud}{\,\mathrm{d}}
\newcommand{\id}{\mathrm{id}}

\DeclareMathOperator{\poly}{poly}

\DeclareMathOperator{\Lip}{Lip}

\DeclareMathOperator{\q}{c}

\DeclareMathOperator{\ns}{X}
\DeclareMathOperator{\nss}{Y}
\DeclareMathOperator{\nh}{H}

\newcommand{\nil}{\mathrm{nil}}
\newcommand{\sml}{\mathrm{sml}}
\newcommand{\unf}{\mathrm{unf}}

\newcommand{\Zmod}[1]{\Z_{#1}} 

\let\originalleft\left
\let\originalright\right
\renewcommand{\left}{\mathopen{}\mathclose\bgroup\originalleft}
\renewcommand{\right}{\aftergroup\egroup\originalright}

\providecommand{\abs}[1]{\left\lvert #1 \right\rvert}
\providecommand{\norm}[1]{\left\lVert #1 \right\rVert}

\providecommand{\Sol}{S}

\providecommand{\eqm}{g}
\providecommand{\lf}{\psi}
\providecommand{\Lf}{\Psi}

\renewcommand{\subset}{\subseteq}

\begin{document}

\title[A continuous model for systems of complexity 2]{A continuous model for systems of complexity 2 on simple abelian groups}
\author{Pablo Candela}
\email{pablo.candela@uam.es}

\author{Bal\'azs Szegedy}
\email{szegedyb@gmail.com}

\date{}
\subjclass[2010]{Primary 11B30; Secondary 11K70, 43A85}
\keywords{Linear forms, higher-order Fourier analysis, nilmanifolds, equidistribution}

\maketitle
\begin{abstract}
It is known that if $p$ is a sufficiently large prime then for every function $f:\Zmod{p}\to [0,1]$ there exists a continuous function on the circle $f':\T\to [0,1]$ such that the averages of $f$ and $f'$ across any prescribed system of linear forms of complexity 1 differ by at most $\epsilon$. This result follows from work of Sisask, building on Fourier-analytic arguments of Croot that answered a question of Green. We generalize this result to systems of complexity at most 2, replacing $\T$ with the torus $\T^2$ equipped with a specific filtration. To this end we use a notion of modelling for filtered nilmanifolds, that we define in terms of equidistributed maps, and we combine this with tools of quadratic Fourier analysis. Our results yield expressions on the torus for limits of combinatorial quantities involving systems of complexity 2 on $\Zmod{p}$. For instance, let $m_4(\alpha,\Zmod{p})$ denote the minimum, over all sets $A\subset \Zmod{p}$ of cardinality at least $\alpha p$, of the density of 4-term arithmetic progressions inside $A$. We show that $\lim_{p\to \infty} m_4(\alpha,\Zmod{p})$ is equal to the infimum, over all continuous functions $f:\T^2\to [0,1]$ with $\int_{\T^2}f\geq \alpha$, of the following integral:
\[
\int_{\T^5} f\binom{x_1}{y_1}\; f\binom{x_1+x_2}{y_1+y_2}\;
 f\binom{x_1+2x_2}{y_1+2y_2+y_3}\; f\binom{x_1+3 x_2}{y_1+3y_2+3y_3} \ud\mu_{\T^5}(x_1,x_2,y_1,y_2,y_3).
\]
\end{abstract}

\section{Introduction}

\noindent One of the well-known central objectives in arithmetic combinatorics is to find optimal bounds for Szemer\'edi's theorem on arithmetic progressions. A closely related problem is to determine how small the average across $k$-term progressions can be for a $[0,1]$-valued function of fixed average value on a large cyclic group $\Z_N=\Z/N\Z$. We write this quantity as follows:\footnote{For a finite set $X$ and a function $f:X\to \C$, we denote by $\E_{x\in X} f(x)$ the average $\frac{1}{|X|}\sum_{x\in X}f(x)$.}
\[
m_k(\alpha,\Zmod{N}):=\; \inf_{f\colon\Zmod{N}\to [0,1],\; \E_{\Zmod{N}} f \,\geq \alpha }\; \E_{n_1,n_2\in\Zmod{N}} f(n_1)f(n_1+n_2)\cdots f(n_1+(k-1)n_2).
\]
A natural direction in which to gain insight on this difficult problem consists in analysing the asymptotic behaviour of these quantities as $N\to \infty$. Answering a question of Green \cite[Problem 3.1]{C&Lev}, Croot took a first step in this direction for the case $k=3$, by proving the following result \cite[Theorem 1]{croot:3APminconv}.
\begin{theorem}[Croot]\label{thm:croot}
For every $\alpha \in [0,1]$, the sequence $m_3(\alpha,\Zmod{N})$ converges as $N\to \infty$ through the primes.
\end{theorem}
\noindent The main tool in Croot's proof is the Fourier transform on $\Zmod{N}$. 
More recently, using higher-order Fourier analysis, the first named author and Sisask extended this convergence result to $k$-term arithmetic progressions for every $k$, and more generally to all systems of linear forms of finite complexity\footnote{We recall this notion of complexity in Definition \ref{def:size&comp}. Let us recall also that the infimum in $m_k(\alpha,\Zmod{N})$ can be restricted to indicator functions of subsets of $\Zmod{N}$ without affecting its asymptotic behaviour, as explained in \cite[p. 3]{croot:3APminconv} for $k=3$; for $k>3$ one can argue similarly, using \cite[Lemma 8.2]{CS3}.} \cite[Theorem 1.3]{CS3}. Thus, provided the values of $N$ are adequately restricted, the problem of estimating $m_k(\alpha,\Zmod{N})$ has a well-defined asymptotic version, namely to estimate the following limit as a function of $\alpha$:
\begin{equation}\label{eq:mainlim}
\lim_{\substack{N\to \infty\\ N\textrm{ prime}}} m_k(\alpha,\Zmod{N}),
\end{equation}
and especially to determine the order of magnitude of this function as $\alpha\to 0$.\\
\indent To shed light on this problem, it is natural to seek an expression for the limit in terms of an integral over some fixed object, hopefully some compact abelian group as simple as possible. A result in this direction was given for $k=3$ by Sisask, identifying the circle $\T=\R/\Z$ as such a group, as follows \cite[Theorem 4.1.2]{SisaskPhD}. 
\begin{theorem}[Sisask]\label{thm:Sisask}
For every $\alpha \in [0,1]$, we have 
\[
\lim_{p\to \infty} m_3(\alpha,\Zmod{p}) = \inf_{\substack{f\colon \T\to [0,1]\;\textrm{Borel}\\ \int_{\T}\; f \,\geq \,\alpha}}\;\; \int_{\T^2} f(x_1)\,f(x_1+x_2)\,f(x_1+2x_2)\; \ud\mu_{\T^2}(x_1,x_2).
\]
\end{theorem}
\noindent (Here and below, given a compact abelian group $Z$ we denote by $\mu_Z$ the Haar probability measure on $Z$, and $p$ denotes a prime number.)

We are thus led to the question of what would be an adequate generalization of Theorem \ref{thm:Sisask} for $k$-term progressions with $k>3$. The main result of this paper provides an answer for $k=4$. In order to state our result let us gather some terminology.
\begin{defn}[$\Lf$-average over a compact abelian group]\label{def:SolAb}
Let $D,t$ be positive integers and let $\Lf=(\lf_1,\ldots,\lf_t)$ be a system of integer linear forms, thus $\lf_i:\Z^D\to\Z$ is a homomorphism for each $i\in [t]$. For a compact abelian group $Z$ and a measurable function $f:Z\to \C$, we define\footnote{Each map $\lf_i$ is originally defined on $\Z^D$ but the definition then extends to any $Z^D$. Thus on $\Z^D$ we have $\lf_i(n)= c_{i,1}n_1+\cdots +c_{i,D}n_D$ for some $c_{i,j}\in \Z$, and then $\lf_i(x)=c_{i,1}x_1+\cdots +c_{i,D}x_D$ on $Z^D$.}
\begin{equation}\label{eq:SolAb}
\Sol_\Lf(f:Z)= \int_{Z^D} f\big(\lf_1(x)\big)\,\cdots\, f\big(\lf_t(x)\big)\, \ud\mu_{Z^D}(x).
\end{equation}
We call this the $\Lf$\emph{-average of} $f$ (or average of $f$ across $\Lf$) over $Z$.
\end{defn}
\noindent We need an analogous notion of averaging across systems of linear forms for functions on a \emph{filtered nilmanifold} $\ns=(G/\Gamma,G_\bullet)$. This analogue relies on the concept of the \emph{Leibman nilmanifold} associated with a system $\Lf:\Z^D\to \Z^t$ and $\ns$. This is a certain subnilmanifold of $\ns^t$ that we shall denote by $\ns^\Lf$. These concepts have been used in previous works (see \cite{GTarith}) and we recall their definitions in Appendix \ref{App}. For now let us take these concepts for granted to give the following main definition.

\begin{defn}[$\Lf$-average over a filtered nilmanifold]\label{def:SolNil}
Let $D,t$ be positive integers and let $\Lf=(\lf_1,\ldots,\lf_t)$ be a system of linear forms $\lf_1,\ldots,\lf_t:\Z^D\to\Z$. Let $\ns=(G/\Gamma,G_\bullet)$ be a filtered nilmanifold (see Definition \ref{def:Filnil}), let $\ns^\Lf$ be the Leibman nilmanifold for $\Lf$ on $\ns$, and let $f:G/\Gamma\to \C$ be measurable.\footnote{Measurability in this paper is always with respect to the Borel $\sigma$-algebra. Recall also that every  nilmanifold $\ns=G/\Gamma$ has a $G$-invariant Borel probability measure, which we denote by $\mu_{\ns}$.} We define the $\Lf$\emph{-average} of $f$ over $\ns$ to be
\begin{equation}\label{eq:SolNil}
\Sol_\Lf(f:\ns)= \int_{\ns^\Lf} f(y_1)\,\cdots\, f(y_t)\, \ud\mu_{\ns^\Lf} (y).
\end{equation}
We define $m_{\Lf}(\alpha,\ns)= \inf\,\big\{S_{\Lf}(f:\ns)~|~f:\ns\to [0,1]\textrm{ Borel},\;  \int_{\ns} f \ud\mu_{\ns}\geq \alpha\big\}$. \footnote{If each coordinate projection $\ns^t\to \ns$ restricts to a measure preserving map from $\ns^\Lf$ to $\ns$, then by a simple argument using Lusin's theorem (discussed in Remark \ref{rem:L1cont} below) one sees that the value of $m_\Lf(\alpha,\ns)$ is unchanged if we restrict the infimum to continuous functions $f:\ns\to [0,1]$.}
\end{defn}
\noindent As a first example, one can check that letting $\ns$ denote $\T$ equipped with the lower central series (this is a nilmanifold that happens to be an abelian group), then, using Definitions \ref{def:LeibGp} and \ref{def:Leibnil}, one has that $\Sol_\Lf(f:\ns)$ is just the average $\Sol_\Lf(f:\T)$ in \eqref{eq:SolAb}. \\ 
\indent An important aspect of Definition \ref{def:SolNil} is that the choice of filtration in $\ns$ can modify the Leibman nilmanifold $\ns^\Lf$ significantly and thereby change the $\Lf$-average (we illustrate this in Example \ref{ex:4AP} below).\\
\indent From now on, given a group $G$, we shall denote by $G_{\bullet(1)}$ the lower central series on $G$. Denoting by $\ns_1$ the filtered nilmanifold $(\T,\R_{\bullet(1)})$, Theorem \ref{thm:Sisask} can be rephrased as saying that $\lim_{p\to \infty} m_{\Lf_3}(\alpha,\Zmod{p}) = m_{\Lf_3}(\alpha,\ns_1)$, where $\Lf_3$ is the system corresponding to 3-term arithmetic progressions. It is known that this result holds more generally in that $\Lf_3$ may be replaced with any system of complexity 1. The point of our main result is that by replacing $\ns_1$ with a natural and only slightly more complicated filtered group, one obtains an extension of Theorem \ref{thm:Sisask} for all systems of complexity at most 2.\\
\indent If $G_{\bullet},H_{\bullet}$ are filtrations on groups $G,H$, with $i$-th terms $G_{(i)},H_{(i)}$, then the \emph{product filtration}, denoted $G_{\bullet}\times H_{\bullet}$, is the filtration on the direct product $G\times H$ with $i$-th term $G_{(i)}\times H_{(i)}$. For an integer $d\geq 2$ and an abelian group $Z$, we denote by $Z_{\bullet(d)}$ the \emph{maximal degree-$d$ filtration} on $Z$, that is the filtration with $i$-th term equal to $Z$ for $i=0,\dots,d$ and equal to $\{0_Z\}$ for $i>d$. We can now state our main result.
\begin{theorem}\label{thm:main}
Let $\ns_2=(\T^2,\;\R_{\bullet(1)}\times\R_{\bullet(2)})$. Then for every system $\Lf$ of integer linear forms of complexity at most 2, and every $\alpha \in [0,1]$, we have
\begin{equation}\label{eq:main}
\lim_{p\to \infty} m_\Lf(\alpha,\Zmod{p}) = m_\Lf(\alpha,\ns_2).
\end{equation}
\end{theorem}
\noindent To illustrate both the result and some of the underlying notions, let us pause to give a more explicit description of \eqref{eq:main} in a specific case.
\begin{example}\label{ex:4AP}
Consider the system $\Lf_4:n\in \Z^2 \mapsto (n_1,n_1+n_2,n_1+2n_2,n_1+3n_2)$ corresponding to 4-term arithmetic progressions, a central example of a system of complexity 2. On $\ns_1$, the Leibman nilmanifold corresponding to $\Lf_4$ is the `usual' 2-dimensional subgroup of $\T^4$ consisting of all 4-term progressions: $\ns_1^{\Lf_4}=\{(x_1,x_1+x_2,x_1+2x_2,x_1+3x_2):x_1,x_2\in \T\}$. If instead we equip $\T$ with the filtration $\R_{\bullet(2)}$, letting $\nss=(\T,\R_{\bullet(2)})$, then the corresponding Leibman nilmanifold for $\Lf_4$ is the following 3-dimensional subgroup of $\T^4$, which contains $\ns_1^{\Lf_4}$: $\nss^{\Lf_4}=\{(y_1, y_1+y_2, y_1+2y_2+y_3, y_1+3y_2+3y_3): y_i\in \T\}$. (Similar calculations have been used as examples in previous works; see in particular \cite[Section 3]{GTarith}.)\\
\indent Now, on $\ns_2$, the Leibman nilmanifold for $\Lf_4$ is $\ns_1^{\Lf_4}\times \nss^{\Lf_4}$, a 5-dimensional subtorus of $\ns_2^4$. Then, Theorem \ref{thm:main} tells us that $\lim_{p\to\infty} m_{\Lf_4}(\alpha,\Zmod{p})$ is equal to the infimum, taken over all continuous functions $f:\T^2\to [0,1]$ with $\int_{\T^2} f \geq \alpha$, of the following integral:
\[
\int_{\T^5} f\binom{x_1}{y_1}\; f\binom{x_1+x_2}{y_1+y_2}\;
 f\binom{x_1+2x_2}{y_1+2y_2+y_3}\; f\binom{x_1+3 x_2}{y_1+3y_2+3y_3} \ud\mu_{\T^5}(x_1,x_2,y_1,y_2,y_3),
\]
where we write elements $\theta$ of $\T^2$ in column form, $\theta=\binom{\theta_1}{\theta_2}$.
\end{example}
\bigskip
\noindent Our proof of Theorem \ref{thm:main} relies on a \emph{transference result} for systems of complexity 2, stated in  Section \ref{sec:transf} (see Theorem \ref{thm:comp2trans}). Roughly speaking, the main part of this result tells us that given a $[0,1]$-valued function $f_1$ on a large group $\Zmod{p}$, there is a continuous $[0,1]$-valued function $f_2$ on $\ns_2$ that resembles $f_1$ in the sense that certain averages of $f_2$ across systems of complexity at most 2 have roughly the same value as those of $f_1$. To motivate our choice of $\ns_2$ further, in Subsection \ref{subsec:counterex} we show that, if $f_1,f_2$ are allowed to be complex-valued, then such a result does not hold if instead of $\ns_2$ we use $\ns_1$.

\noindent The proof of our transference result occupies most of the rest of the paper. It relies on the regularity method for the Gowers $U^3$ norm, and also on a notion of \emph{modelling} between filtered nilmanifolds that is central to this paper, defined in Subsection \ref{subsec:outline} (see Definitions \ref{def:multib} and \ref{def:modelling}). Proving the transference result then consists essentially in showing that $\ns_2$ models every filtered nilmanifold of degree at most 2 (see Theorem  \ref{thm:maincasereduc} and Proposition \ref{prop:mcr-implies-ms}). To prove the latter statement, we use in particular a certain decomposition of general filtered nilmanifolds of degree at most 2, which we establish in Section \ref{sec:decomp} (see Proposition \ref{prop:prodecompsoft}).\\
\indent In Section \ref{sec:h3}, we rule out an a-priori plausible approach to generalizing Theorem \ref{thm:main} for higher-complexity systems, by showing that the degree-3 filtered torus $(\T^3,\R_{\bullet(1)}\times \R_{\bullet(2)}\times \R_{\bullet(3)})$ does not model all filtered nilmanifolds of degree at most 3. The counterexample involves the Heisenberg nilmanifold with a certain degree-3 filtration (see Proposition \ref{prop:countercomp3}). This example shows that the possibility to use a torus of this kind in a transference result such as Theorem \ref{thm:comp2trans} is a specific feature of systems of complexity at most 2. This phenomenon provides a new illustration of the marked increase in difficulty in passing from systems of complexity 2 to sytems of complexity at least 3 (this increase has been observed in different contexts before, for example in \cite[\S 11]{GSz}).\\

\section{A transference result for systems of complexity 2}\label{sec:transf}

\bigskip

\noindent To prove the convergence of quantities such as $m_k(\alpha,\Zmod{p})$, it is natural to try to transfer a function on a large group $\Zmod{p}$ to another such group $\Zmod{q}$ while keeping certain $\Lf$-averages of this function roughly unchanged. This idea is already present in Croot's proof of Theorem \ref{thm:croot} in \cite{croot:3APminconv}, where it is implemented by an argument that uses Fourier analysis and that is extendable to any system of complexity 1 (in \cite{croot:3APminconv} it is used just for the central case of 3-term arithmetic progressions).

Extending this method towards the continuous setting, results such as Theorem \ref{thm:Sisask} can be deduced from transference results that are similar except that now one has to allow either of $\Zmod{p}$ or $\Zmod{q}$ to be replaced with $\T$. This was originally done in \cite{SisaskPhD}, also using Fourier analysis (see \cite[Proposition 4.2.8]{SisaskPhD}, and also \cite[Corollary 3.8]{CS1}).

Theorem \ref{thm:main} can be obtained in a similar spirit as an immediate consequence of the following result, but our proof of this result will use \emph{quadratic} Fourier analysis.\\

\begin{theorem}[Transference for systems of complexity 2 between $\Zmod{p}$ and $\ns_2$]\label{thm:comp2trans}\hfill\\
\noindent For every $\epsilon > 0$, there exists $C>0$ such that the following holds. Let $\ns=\Zmod{p}$ with $p \geq C$ and let $\ns'=\ns_2$. Then for every function $f : \ns \to [0,1]$, there exists a continuous function $f' : \ns' \to [0,1]$ such that, for every system $\Lf$ of linear forms of size\footnote{We recall this notion of size in Definition \ref{def:size&comp}.} at most $1/\epsilon$ and complexity at most 2, we have $\abs{ S_{\Lf}(f : \ns) - S_\Lf(f' : \ns') } \leq \epsilon$. The last sentence holds also with $\ns=\ns_2$ \textup{(}with $f$ being measurable\textup{)} and $\ns'=\Zmod{p}$.
\end{theorem}
\noindent We split the proof into two cases. The main case $\ns=\Zmod{p}$, $\ns'=\ns_2$ will occupy most of the sequel. The other case, namely $\ns=\ns_2$, $\ns'=\Zmod{p}$, is simpler and is treated briefly in Section \ref{sec:LC-maps}; see Proposition \ref{prop:easycase}. We shall end this section with an outline of the proof of the main case, in Subsection \ref{subsec:outline}. Before that, however, we pause to give a sense in which the nilmanifold $\ns_2$ in Theorem \ref{thm:comp2trans} cannot be simplified.

\subsection{Some counterexamples}\label{subsec:counterex}\hfill

\noindent One may call a filtered nilmanifold playing the role of $\ns_2$ in Theorem \ref{thm:comp2trans} a \emph{continuous model} for systems of complexity at most 2 on groups $\Zmod{p}$. One might then wonder whether simpler models than $\ns_2$ could be used, and in particular whether one of the two components of $\ns_2$, namely $(\T,\R_{\bullet(1)})$, $(\T,\R_{\bullet(2)})$, could already suffice. In this subsection we show that, if in Theorem \ref{thm:comp2trans} we allow complex-valued functions and also averages involving complex conjugation, then the theorem does not hold with any of these two components.\\
\indent We start by ruling out $(\T,\R_{\bullet(1)})=\ns_1$. To this end we shall use the systems of linear forms related to the Gowers norms  $\|\cdot\|_{U^2}$ and $\|\cdot\|_{U^3}$. Let
\begin{equation}\label{eq:U2}
\Lf_{U^2}\colon n\in \Z^3 \;\;  \mapsto \;\; (n_1,\;n_1+n_2,\;n_1+n_3,\;n_1+n_2+n_3),
\end{equation}
and let
\begin{eqnarray}\label{eq:U3}
\Lf_{U^3}\colon n\in \Z^4 & \mapsto & (n_1,\;\;n_1+n_2,\;\;n_1+n_3,\;\;n_1+n_2+n_3,\\
&&\hspace{0.5cm}n_1+n_4,\;\;n_1+n_2+n_4,\;\;n_1+n_3+n_4,\;\;n_1+n_2+n_3+n_4).\nonumber
\end{eqnarray}
For a bounded measurable function $f:Z\to \C$ on a compact abelian group $Z$, we set
\[
\Sol_{\Lf_{U^2}}'(f:Z)=\int_{Z^3}\;\; f(x_1)\;\overline{f(x_1+x_2)}\;\overline{f(x_1+x_3)}\;f(x_1+x_2+x_3)\;\; \ud\mu_{Z^3}(x_1,x_2,x_3),
\]
and we define $\Sol_{\Lf_{U^3}}'(f:Z)$ similarly (in particular $\Sol_{\Lf_{U^d}}'(f:Z)=\|f\|_{U^d(Z)}^{2^d}$, $d=2,3$).\\ 
\indent We then have the following result, which shows that Theorem \ref{thm:comp2trans} does not hold for all functions taking values in the unit disc $\cD\subset \C$ if we replace $\ns_2$ with $\ns_1$.
\begin{proposition}\label{prop:deg-1-circle-out}
For each prime $p$, let $f_p:\Zmod{p}\to \cD$, $x\mapsto e(x^2/p)=\exp(2\pi i\,x^2/p)$. There exists an absolute constant $\delta>0$ such that, for every $p$  sufficiently large, there is no continuous function $f:\T\to \cD$ such that
\begin{equation}\label{eq:counterex}
\max\left \{\;\Big|\,\Sol_{\Lf_{U^2}}'(f_p:\Zmod{p})- \Sol_{\Lf_{U^2}}'(f:\T)\,\Big|\,,\; \Big|\,\Sol_{\Lf_{U^3}}'(f_p:\Zmod{p})- \Sol_{\Lf_{U^3}}'(f:\T)\,\Big|\;\right\}\; <\; \delta.
\end{equation}
\end{proposition}
\noindent The function $f_p$ in this result has been used before as a source of examples in the discrete setting (see \cite[\S 4]{GSz}). In particular, a standard calculation shows that $\Sol_{\Lf_{U^3}}'(f_p:\Zmod{p})= 1$ while $\Sol_{\Lf_{U^2}}'(f_p:\Zmod{p})= o(1)_{p\to \infty}$ (see \cite[Exercise 11.1.12]{T-V}). We want to show that this behaviour cannot be reproduced by a continuous $\cD$-valued function on $\T$. This will follow  from a combination of two results.\\
\indent Given compact abelian groups $Z,Z'$, we say that a continuous map $P:Z\to Z'$ is a \emph{polynomial map of degree $\leq d$} if
\begin{equation}\label{eq:abpolymap}
\Delta_{h_1}\cdots \Delta_{h_{d+1}} P(x)=0\;\textrm{ for every }x,h_1,\ldots,h_{d+1}\in Z,
\end{equation}
where $\Delta_h P(x):= P(x+h)-P(x)$.

There is a related definition of polynomial maps between filtered groups, that we recall in Appendix \ref{App}; see Definition \ref{def:polymaps}, after which we also recall the relation between the two definitions.

Of the two results that we shall combine as mentioned above, the first  is the following characterization, due to Eisner and Tao, of bounded functions that are nearly extremal for the $U^d$ norm  \cite[Theorem 1.1]{Eisner-Tao}.
\begin{theorem}[$L^\infty$ near-extremisers on compact abelian groups]\label{thm:ET}  Let $d$ be a positive integer, let $Z$ be a compact abelian group, and let $f \in L^\infty(Z)$ be such that $\|f\|_{L^\infty(Z)} \leq 1$.  If $\|f\|_{U^d(Z)} \geq 1-\epsilon$, then there exists a continuous polynomial map $P: Z \to \R/\Z$ of degree $\leq d-1$ such that $\|f-e(P)\|_{L^1(Z)} = o(1)_{\epsilon \to 0}$.
\end{theorem}
\noindent The second result concerns continuous polynomial maps between tori. (A special case appeared already in work of the second named author; see \cite[Lemma 1.3]{Szegedy:HFA}.) 
\begin{lemma}\label{lem:ctspolylin} Let $m\geq 1$ and fix any filtrations $\T^m_\bullet,\T_\bullet$ on $\T^m,\T$. Then every continuous polynomial map $f:(\T^m,\T^m_\bullet) \to (\T,\T_\bullet)$ is a polynomial map $\T^m\to \T$ of degree $\leq 1$.
\end{lemma}
\begin{proof} Note first that the map $f$, which is assumed to be polynomial in the sense of Definition \ref{def:polymaps}, is then also polynomial of degree at most $d$ in the sense of \eqref{eq:abpolymap}, where $d$ is the degree of the filtration  $\T_\bullet$. Now suppose for a contradiction that $f$ is not of degree $\leq 1$. Then by repeated applications of difference operators $\Delta_h$ to $f$, we eventually obtain a map $g$ such that $\Delta_{h_3}\Delta_{h_2}\Delta_{h_1}g(x)=0$ for all $x,h_1,h_2,h_3\in \T^m$, but $\Delta_{h_2}\Delta_{h_1}g(x)$ is not zero for all $x,h_1,h_2$. It therefore suffices to obtain a contradiction if $f$ is such a map. In other words, we may suppose that $f$ is polynomial of degree $\leq 2$ but not of degree $\leq 1$. Now note that, in this case, for each fixed $h$ the map $\Delta_h f$ is polynomial of degree $\leq 1$. Setting $\theta_h=\Delta_h f(0)$, we therefore have that $x\mapsto \Delta_h f(x)-\theta_h$ is a continuous homomorphism $\T^m\to \T$.  Hence there is $n_h\in \Z^m$ such that $\Delta_h f(x)=n_h\cdot x+\theta_h$.  (Here $n\cdot x=n(1)x(1)+\cdots+n(m) x(m)$.) On the other hand, the uniform continuity of $f$ implies that the map $h\mapsto \Delta_h f-\theta_h$ is continuous from $\T^m$ to the space of continuous functions $\T^m\to\T$. (On the latter space we use the topology given by the metric $d(f,g)=\sup_{x\in \T^m} \|f(x)-g(x)\|_\T$.) 
The continuous homomorphisms $\T^m\to\T$, being of the form $x\mapsto n\cdot x$ for some $n\in\Z^m$, form a discrete set in this space. Since $\T^m$ is connected, the map $h\mapsto n_h$ must therefore be constant,  so there is $n\in \Z^m$ such that $\Delta_h f(x)= n\cdot x+\theta_h$ for all $h,x$. Applying this with $h=0$ implies that $n\cdot x+\theta_0=0$ for all $x\in \T^m$, which implies that $n=0$. Thus $\Delta_h f(x)=\theta_h$ for all $h,x$, and so $\Delta_k \Delta_h f(x)=0$ for all $h,k,x$, which contradicts that $f$ is not polynomial of degree $\leq 1$.
\end{proof}

\begin{proof}[Proof of Proposition \ref{prop:deg-1-circle-out}]
If $f:\T\to \cD$ is a continuous function satisfying
\[
\Big|\,\Sol_{\Lf_{U^3}}'(f_p:\Zmod{p})- \Sol_{\Lf_{U^3}}'(f:\T)\,\Big|<\delta,
\]
then by Theorem \ref{thm:ET} together with the fact that $\Sol_{\Lf_{U^3}}'(f_p:\Zmod{p})=1$, there is some continuous polynomial map $P:\T\to \T$ of degree $\leq 2$ such that $\|f-e(P)\|_{L^1(\T)}\leq \epsilon$, where $\epsilon=\epsilon(\delta)\to 0$ as $\delta\to 0$. By Lemma \ref{lem:ctspolylin} we have in fact that $P$ is of degree $\leq 1$. But then a simple calculation shows that $\Sol_{\Lf_{U^2}}'(f:\T)= \Sol_{\Lf_{U^2}}'(e(P):\T)-o(1)=1-o(1)$ as $\delta\to 0$. Since $\Sol_{\Lf_{U^2}}'(f_p:\Zmod{p})=o(1)_{p\to\infty}$, the inequality $\Big|\,\Sol_{\Lf_{U^2}}'(f_p:\Zmod{p})- \Sol_{\Lf_{U^2}}'(f:\T)\,\Big|<\delta$ fails for $\delta$ sufficiently small and $p$ sufficiently large.
\end{proof}

\begin{remark}\label{rem:Ynot}
Ruling out $\nss=(\T,\R_{\bullet(2)})$ is a much simpler task. It suffices to use the trivial form $\Lf_0:n\in \Z\mapsto n$ (note that $\Sol_{\Lf_0}(f:Z)=\int_Z f\ud\mu_Z$ for any compact abelian group $Z$) and some system of complexity 1, say $\Lf:n\in \Z^2\mapsto (n_1,n_2,n_1+n_2)$. Indeed, we can find a set $A_p\subset \Zmod{p}$ such that for large $p$ the following inequality fails for every continuous function $f:\T\to[0,1]$:
\[
\max\left \{\;\Big|\,\Sol_{\Lf_0}(1_{A_p}:\Zmod{p})- \Sol_{\Lf_0}(f:\nss)\,\Big|\,,\; \Big|\,\Sol_{\Lf}(1_{A_p}:\Zmod{p})- \Sol_{\Lf}(f:\nss)\,\Big|\;\right\}\; <\; 1/100.
\]
To see this, note first that $\nss^\Lf=\T^3$ (using Definitions \ref{def:LeibGp} and \ref{def:Leibnil}), so any such function $f$ satisfies $\Sol_\Lf(f:\nss)=(\int_\T f)^3=\Sol_{\Lf_0}(f:\nss)^3$. By contrast, for a function $f_p:\Zmod{p}\to [0,1]$, the average $S_\Lf(f_p:\Zmod{p})=\E_{n_1,n_2\in \Zmod{p}}f(n_1)f(n_2)f(n_1+n_2)$ can be much smaller than $\Sol_{\Lf_0}(f_p:\Zmod{p})^3$. For instance, let $f_p=1_{A_p}$ where $A_p$ is the large sumfree set $(p/3,2p/3)\subset \Zmod{p}$.
\end{remark}
\begin{remark}
It would be interesting to establish stronger variants of Proposition \ref{prop:deg-1-circle-out} that would rule out $\ns_1$ as a continuous model for systems of complexity 2 more decisively. For instance, one can ask whether in that proposition $e(x^2/p)$ can be replaced with some $[0,1]$-valued function. Such a real-valued counterexample could require allowing more than two systems of complexity at most 2 in the maximum in \eqref{eq:counterex}. In closer relation to Theorem \ref{thm:main}, it would be conclusive to find a system $\Lf$ of complexity 2 for which $\lim_{p\to \infty} m_\Lf(\alpha,\Zmod{p})< m_\Lf(\alpha,\ns_1)$ for some $\alpha$. In particular, we do not know whether this inequality  holds for the system $\Lf_4$ of 4-term arithmetic progressions.
\end{remark}
\begin{question}
Does the inequality  $m_{\Lf_4}(\alpha,\ns_2)< m_{\Lf_4}(\alpha,\ns_1)$ hold for some $\alpha \in [0,1]$?
\end{question}
\noindent As a counterpoint to the last remark, let us mention that in the next subsection we shall introduce a definition of modelling, concerning filtered nilmanifolds, and that with respect to this definition $\ns_2$ can definitely not be replaced with $\ns_1$ (as detailed in Remark \ref{rem:X1not} below). This definition is in some sense more natural than (though related to) the notion of a  continuous model for systems on $\Zmod{p}$ considered above.

\subsection{A reduction of the main case of Theorem \ref{thm:comp2trans}}\label{subsec:outline}\hfill

\noindent Our central aim from now on is to establish the following result.
\begin{proposition}\label{prop:maincase}
For every $\epsilon >0$, there exists $C>0$ such that the following holds. For every function $f: \Zmod{p}\to [0,1]$ with $p\geq C$, there exists a continuous function $f' : \T\times \T \to [0,1]$ such that, for every system $\Lf$ of linear forms of size at most $1/\epsilon$ and complexity at most 2, we have $\abs{ S_{\Lf}(f': \ns_2) - S_\Lf(f: \Zmod{p}) } \leq \epsilon$. 
\end{proposition}
\begin{remark}\label{rem:L1cont}
This result is equivalent to the seemingly weaker version in which $f'$ is only claimed to be measurable. (The version in terms of measurable functions is also natural in that it accommodates the indicator functions of measurable sets.) The equivalence can be seen using Lusin's theorem to approximate $f'$ in $L^1(\T^2)$ by a continuous function $h$, and then the fact (which will also be used later) that for all measurable functions $f',h$ on $\T^2$ bounded by 1 and every system $\Lf:\Z^D\to \Z^t$ of finite complexity, we have
\[
\abs{ S_{\Lf}(f': \ns_2) -S_{\Lf}(h: \ns_2) } \leq t\, \norm{f'-h}_{L^1(\T^2)}.
\]
This $L^1$-continuity of $S_\Lf(\cdot:\ns_2)$ follows from the fact that any coordinate projection $(\T^2)^t\to\T^2$ restricts to a surjective map from $\ns_2^{\Lf}$ to $\T^2$. (Note that the complexity assumption for $\Lf$ implies that none of the linear forms is the 0-form; see Definition \ref{def:size&comp} and the paragraph thereafter.)
\end{remark}
\noindent For a continuous function $F:X\to \C$ on a metric space $(X,d)$, we define the Lipschitz norm of $F$ by $\|F\|_{\Lip}=\|F\|_{\infty}+\sup_{x,y\in X, x\neq y} \frac{|F(x)-F(y)|}{d(x,y)}$.

\begin{defn}[Quantitative equidistribution]\label{def:equid}
Let $(\ns,\mu)$ be a probability space and let $\ns'$ be a metric space with a Borel probability measure $\mu'$. We say that a measurable map $\phi:\ns \to \ns'$ is \emph{$\delta$-equidistributed} if for every Lipschitz function $F:\ns'\to\C$ we have
\begin{equation}
\left| \int_{\ns} F\circ \phi\; \ud\mu - \int_{\ns'} F \;\ud\mu'\right|\leq \delta\, \norm{F}_{\Lip}.
\end{equation}
\end{defn}
\noindent On a filtered nilmanifold with a fixed Mal'cev basis $\X$ (see Definition \ref{def:Filnil}) there is a convenient metric, which was defined in  \cite[Definition 2.2]{GTOrb}. Using this metric we shall now define a notion of equidistribution concerning certain maps between filtered nilmanifolds, which plays a central role in the sequel. This is a strong form of equidistribution in the sense that, in addition to the given map being equidistributed, we require various multiparameter versions of the map, corresponding to Leibman nilmanifolds for systems of linear forms, also to be equidistributed. In the special case of polynomial sequences, this is related to the  stronger notion of \emph{irrationality} from \cite{GTarith} (see also \cite{CS3}).

Given filtered groups $(G,G_\bullet), (G',G_\bullet')$, a polynomial map $\eqm\in \poly_0(G_\bullet,G'_\bullet)$ (see Definition \ref{def:polymaps}), and subgroups $\Gamma,\Gamma'$ of $G,G'$ respectively, we say that $\eqm$ is \emph{$(\Gamma,\Gamma')$-consistent} if for every $\gamma\in \Gamma$ the map $x\mapsto \eqm(x)^{-1}\eqm(x\gamma)$ is $\Gamma'$-valued on $G$. Note that if $g$ is $(\Gamma,\Gamma')$-consistent then $x\Gamma \mapsto g(x)\Gamma'$ is a well-defined map $G/\Gamma\to G'/\Gamma'$.  We  denote this map by $\phi_g$ and say that it is \emph{induced} by $g$. Note also the following composition property: if $g\in \poly_0(G_\bullet,G'_\bullet)$ is $(\Gamma,\Gamma')$-consistent and $g'\in \poly_0(G'_\bullet,G''_\bullet)$ is $(\Gamma',\Gamma'')$-consistent, then $g'\circ g$ is in $\poly_0(G_\bullet,G''_\bullet)$ (by composition of polynomial maps; see the paragraph after Definition \ref{def:polymaps}) and is $(\Gamma,\Gamma'')$-consistent.

The strong notion of equidistribution is the following.
\begin{defn}[Balanced map]\label{def:multib}
Let $\ns=(G/\Gamma,G_\bullet)$ be a possibly disconnected filtered nilmanifold (see Remark \ref{rem:disconnected}), let $\ns'=(G'/\Gamma',G'_\bullet,\X')$ be a filtered based nilmanifold (see Definition \ref{def:Filnil}), and let $\eqm:G\to G'$ be a continuous $(\Gamma,\Gamma')$-consistent map in $\poly_0(G_\bullet,G'_\bullet)$. For $\delta>0$ we say that $\phi_\eqm:\ns\to\ns'$ is $\delta$\emph{-balanced} if for each system $\Lf:\Z^D\to\Z^t$ of linear forms of size at most $1/\delta$ the following map is $\delta$-equidistributed:\footnote{There is a slight abuse of notation in \eqref{eq:multib} in that we are implictly using the identification of a nilmanifold $\ns^\Lf$ with a subnilmanifold of $\ns^t$, via the embedding $x\Gamma^\Lf\mapsto x\Gamma^t$. Note also that the metric used on ${\ns'}^{\,\Lf}$ for $\delta$-equidistribution here is the restriction of the metric on ${\ns'}^{\,t}$ related to the Mal'cev basis $\X'$; see \cite[Appendix A]{CS3}.}
\begin{equation}\label{eq:multib}
\phi_g^t\colon \ns^\Lf\;\;\to\;\; {\ns'}^{\,\Lf},\quad x\Gamma^t\;\;\mapsto\;\; \eqm^t(x){\Gamma'}^{\,t}= \big(g(x_1),\dots,g(x_t)\big){\Gamma'}^{\,t}.
\end{equation}
\end{defn}

\begin{remark} We allow $\ns$ to be possibly disconnected here (in the sense of Remark \ref{rem:disconnected}) so that the notion of a balanced map can concern maps defined on a filtered finite abelian group $\ns$. Thus, for instance, if $G=\Z$ with the lower central series  and $\Gamma=p\,\Z$, we have that $g:G\to G'$ is $(\Gamma,\Gamma')$-consistent if and only if $g$ is a polynomial sequence in $\poly_0(\Z,G'_\bullet)$ that is \emph{$p$-periodic mod $\Gamma'$} in the sense of \cite[Definition 2.10]{CS3}. In this case,  note that $\phi_g$ is $\delta$-balanced if the polynomial $g$ is sufficiently \emph{irrational} in the sense of \cite[Definition 4.7]{CS3}.
\end{remark}

\noindent This notion of a balanced map is related to that of a \emph{balanced morphism}, a natural analogue in the category of nilspaces (see  \cite[Definition 1.2]{Szegedy:HFA}); we shall not detail this relation in this paper.\\
\indent In Definition \ref{def:multib} we are implicitly using the property that $\eqm^t(x)$ is in the Leibman group $(G',G_\bullet')^{\Lf}$ whenever $x\in (G,G_\bullet)^\Lf$. This property is nontrivial but a short proof can be given using known facts (see Proposition \ref{prop:polykeepLeib}).\\
\indent Note also that if $\phi_\eqm: \ns\to\ns'$ is $\delta$-balanced then in particular, using the relation between the metrics corresponding to $\X'$ and ${\X'}^t$ (see \cite[Lemmas A.3 and A.4]{CS3}) we have, for every Lipschitz function $F:\ns'\to\C$,
\begin{equation}\label{eq:ProdLipBnd}
\left| \Sol_\Lf(F\circ \phi_\eqm:\ns) - \Sol_\Lf(F:\ns')\right|\leq \delta \big\|F^{\otimes^t}\big\|_{\Lip({\X'}^t)}\leq \delta\,t\,\norm{F}_{\Lip(\X')}^t.
\end{equation}
(Here $F^{\otimes^t}$ denotes the function $X^t\to \C$, $(x_1,\dots,x_t)\Gamma^t\mapsto F(x_1\Gamma)\cdots F(x_t\Gamma)$.)\\
\indent With the notion of balanced maps we can now give one of the central definitions of this paper.
\begin{defn}[Modelling filtered nilmanifolds]\label{def:modelling}
Let $\ns,\ns'$ be filtered nilmanilfolds. We say that $\ns$ \emph{models} $\ns'$ if there exists a Mal'cev basis $\X'$ on $\ns'$ such that for every $\delta>0$ there exists a $\delta$-balanced map $\phi_g:\ns\to\ns'$.
\end{defn}
\noindent The basis $\X'$ is used just to have the metric structure underlying the notion of balanced map, but note that if the main statement holds for one such basis then it holds for any other. One can see this using results relating the metrics corresponding to two given bases; see for instance \cite[Lemma A.17]{GTOrb}. Note also that modelling induces a preorder on the set of filtered nilmanifolds. We discuss this further at the end of Section \ref{sec:remarks}.\\
\indent The first main tool that we use to prove Proposition \ref{prop:maincase} is the following consequence of the regularity method for the Gowers $U^3$ norm. 
\begin{theorem}\label{thm:Zp-to-X}
For every $\delta >0$ there exists $C>0$ such that the following holds. For every function $f : \Zmod{p} \to [0,1]$ with $p\geq C$, there is a filtered based nilmanifold $\ns=(G/\Gamma, G_\bullet, \X)$ of degree at most $2$ and complexity at most $C$, and a continuous function $F : G/\Gamma \to [0,1]$ with $\norm{F}_{\Lip(\X)} \leq C$, such that for every system $\Lf$ of integer linear forms of size at most $1/\delta$ and complexity at most 2, we have
\begin{equation}\label{eq:Zp-to-X}
\abs{\, \Sol_\Lf(f : \Zmod{p}) - \Sol_\Lf(F: \ns )\,} \leq \delta.
\end{equation}
\end{theorem}
\noindent This was essentially proved in \cite{CS3} using arguments and results from \cite{GTarith,Szegedy:HFA}, but for the above version we need some small modifications, so for completeness we include a proof in Appendix \ref{App} (see Theorem \ref{thm:Zp-to-X-app}).\\
\indent Given Theorem \ref{thm:Zp-to-X}, proving Proposition \ref{prop:maincase} reduces to establishing the following.
\begin{theorem}\label{thm:maincasereduc}
Let $\ns_2=(\T^2,\R_{\bullet(1)}\times \R_{\bullet(2)})$ and let $\ns$ be a filtered nilmanifold of degree at most $2$. Then $\ns_2$ models $\ns$.
\end{theorem}

\begin{proposition}\label{prop:mcr-implies-ms}
Theorem \ref{thm:maincasereduc} implies Proposition \ref{prop:maincase}.
\end{proposition}
\begin{proof}
Given $\epsilon>0$ as in Proposition \ref{prop:maincase}, we first apply Theorem \ref{thm:Zp-to-X}, with parameter $\delta_1$ a function of $\epsilon$ to be fixed later. Thus for every $f:\Zmod{p}\to [0,1]$ with $p\geq C(\delta_1)$ there exists $F:\ns \to [0,1]$ with $\|F\|_{\Lip(\X)}\leq C(\delta_1)$ such that for every system $\Lf:\Z^D\to\Z^t$ of complexity at most 2 and size at most $1/\delta_1$, we have 
$|\Sol_\Lf(f:\Zmod{p})- \Sol_\Lf(F:\ns)|\leq \delta_1$.\\
\indent Now, by Theorem \ref{thm:maincasereduc}, there exists a $\delta_2$-balanced map $\phi_g:\ns_2\to \ns$ (with respect to the metric on $\ns$ given by $\X$), where $\delta_2>0$ is a function of $\epsilon$ that we fix later.\\
\indent Let $f':\ns_2\to [0,1]$ be the continuous function $F\circ \phi_g$. Then by \eqref{eq:ProdLipBnd} we have
\[
|\Sol_\Lf(F:\ns)-\Sol_\Lf(f':\ns_2)|\leq \delta_2\, t\, \|F\|^t_{\Lip(\X)}\leq \delta_2\,\delta_1^{-1}\,C(\delta_1)^{1/\delta_1}.
\]
Letting $\delta_1=\epsilon/2$ and $\delta_2$ such that $\delta_2\,\delta_1^{-1}\,C(\delta_1)^{1/\delta_1}\leq \epsilon/2$, the result follows.
\end{proof}
\noindent To prove Theorem \ref{thm:maincasereduc}, we shall first replace $\ns$ with a product of two simpler filtered nilmanifolds, namely a 2-step nilmanifold $\nss_1$  with lower central series and a torus $\nss_2$ with the maximal degree-$2$ filtration. This is done in the next section. As a consequence,  finding the desired $\delta$-balanced map $\phi_g$ for Theorem \ref{thm:maincasereduc} will be reduced to finding a sufficiently balanced map $\phi_1: \ns_1 \to\nss_1$ and another such map $\phi_2:(\T,\R_{\bullet(2)})\to \nss_2$ (see Proposition \ref{prop:ReducToTori}). We obtain these maps in Sections \ref{sec:LC-maps} and \ref{sec:QT-maps} respectively.

\begin{remark}\label{rem:X1not}
As mentioned at the end of Subsection \ref{subsec:counterex}, in Theorem  \ref{thm:maincasereduc} one cannot replace $\ns_2$ with $\ns_1$. In fact $\ns_1$ does not even model $\nss=(\T,\R_{\bullet(2)})$. To see this, let $\Lf=\Lf_{U^2}$ as in \eqref{eq:U2}. Computing the Leibman nilmanifolds, we find that $\nss^\Lf=\T^4$, whereas $\ns_1^\Lf$ is the 3-dimensional subtorus $\Lf(\T^3)$ of $\T^4$. Since a continuous polynomial map $\phi_g:\T\to \T$ with $\phi_g(0)=0$ must be a homomorphism (by Lemma \ref{lem:ctspolylin}), we have $\phi_g^4(\Lf(\T^3))\subset \Lf(\T^3)$. Therefore, for $\delta$ sufficiently small, for every such map $\phi_g$ we have that $\phi_g^4$ is not $\delta$-equidistributed in $\nss^\Lf$, and so $\phi_g:\ns_1\to \nss$ is not $\delta$-balanced.
\end{remark}
\section{A product decomposition of degree-2 filtered nilmanifolds}\label{sec:decomp}

\noindent The main result of this section describes the structure of a general filtered nilmanifold of degree 2 by decomposing it into the product of two nilmanifolds with very specific filtrations. The formal statement uses the following notions.

\begin{defn}[Isomorphism of filtered nilmanifolds]\label{def:BNilIso}
Let $\ns=(G/\Gamma,G_\bullet)$ and $\ns'=(G'/\Gamma',G'_\bullet)$ be  filtered nilmanifolds of degree $d$. We say that $\ns$ and $\ns'$ are \emph{isomorphic} if there is a map $\theta:G\to G'$ satisfying the following conditions:
\begin{enumerate}
\item $\theta$ is an isomorphism of filtered Lie groups. This means that for each $i\geq 0$ the restriction of $\theta$ to $G_{(i)}$ is an isomorphism of Lie groups $G_{(i)}\to G_{(i)}'$.
\item We have $\theta(\Gamma)=\Gamma'$.
\end{enumerate}
If Mal'cev bases $\X=\{x_1,\dots,x_r\}$, $\X'=\{x_1',\dots,x_s'\}$ are given on $\ns,\ns'$ respectively, then we say that  $\ns,\ns'$ are \emph{isomorphic} as filtered \emph{based} nilmanifolds if there is a map $\theta:G\to G'$ satisfying condition (i) above (so that in particular we have $r=s$) and satisfying also the following condition:
\begin{enumerate}
\item[(ii')] For each $j\in [r]$ we have $\theta(\exp(x_j))=\exp(x_j')$.
\end{enumerate}
\end{defn}

\medskip

\noindent Note that (i) and (ii') together imply (ii), since by the defining properties of Mal'cev bases we have $\Gamma=\{\exp(t_1 x_1)\cdots \exp(t_m x_m):t_j\in \Z\}$.

\medskip
The structural result can be stated as follows.
\begin{proposition}\label{prop:prodecompsoft}
Let $\ns=(G/\Gamma,G_\bullet)$ be a filtered nilmanifold of degree at most 2. Then $\ns$ is isomorphic to $\nss_1\times \nss_2$, where $\nss_1$ is a 2-step nilmanifold with lower central series, and $\nss_2$ is a torus with the maximal degree-2 filtration.
\end{proposition}
\noindent We shall actually prove a refinement of this result, namely Proposition \ref{prop:prodecomp} below, in which an additional metric structure on $\ns$ given by a fixed Mal'cev basis is also approximately conserved by the isomorphism. This refinement can be useful from a quantitative point of view, especially in relation to the regularity lemma for the $U^3$ norm (see Remark \ref{rem:UseOfStruct}).

\medskip
\indent Given two Mal'cev bases $\X,\X'$ on a filtered nilmanifold, we say that $\X'$ is $Q$-\emph{rational relative to }$\X$ if each element from $\X'$ is a linear combination of the elements of $\X$ with rational coefficients of height at most $Q$. (Similar terminology was used in \cite[Definitions 2.4, 2.5]{GTOrb}.)

\begin{proposition}\label{prop:prodecomp}
There is a function $\eta:\R_{>0}\to \R_{>0}$ such that the following holds. Let $Q\geq 2$, and let $\ns=(G/\Gamma,G_\bullet,\X)$ be of degree 2 and complexity at most $Q$. Then there exists a Mal'cev basis $\X'$ for $(G/\Gamma,G_\bullet)$ that is $\eta(Q)$-rational relative to $\X$, such that $\ns':=(G/\Gamma,G_\bullet,\X')$ is isomorphic \textup{(}as a filtered based nilmanifold\textup{)} to $\nss_1\times \nss_2$, where $\nss_1=(H/\Gamma_H,H_{\bullet(1)},\X_H)$ is a 2-step nilmanifold with lower-central series and complexity at most $\eta(Q)$, and $\nss_2=(\T^m,\R^m_{\bullet(2)})$ with the standard basis on $\R^m$ and with $m\leq Q$. 
\end{proposition}
\noindent As the proof reveals, we can take $\eta(Q)=Q^K$ for some constant $K$  depending only on the dimension and degree of $(G,G_\bullet)$; in particular, $K$ also depends only on $Q$.
\begin{proof}
In this proof, just as in \cite[Appendix A]{GTOrb}, the constants implicit in  $O(1)$ notations are allowed to depend on the dimension and degree of $(G,G_\bullet)$.\\
\indent The basis $\X$ is adapted to the given filtration $G_\bullet=(G_{(0)},G_{(1)},G_{(2)},\{\id_G\},\ldots)$. First, we claim that there is a Mal'cev basis $\X'$ for $G/\Gamma$, which is $Q^{O(1)}$-rational relative to $\X$, and which refines $\X$ in the sense that $\X'$ is also adapted to $G_\bullet$ and in addition it passes through the subgroup $G_{(3)}:=[G,G]\leq G_{(2)}$.\\
\indent To see this, note first that $G_{(3)}$ is a $Q$-rational subgroup of $G$ relative to $\X$, in the sense of \cite[Definition 2.5]{GTOrb}. Indeed, letting $\frak{g}$ denote the Lie algebra of $G$, we have the fact that the Lie subalgebra $\frak{g}_{(3)}$ corresponding to $G_{(3)}$ is spanned by Lie brackets $[x_i,x_j]$ where $x_i,x_j\in \X$.  (By \cite[Proposition 5.2.1]{C&G} we know that $\frak{g}_{(3)}$ is the $\R$-span of elements $[v,w]$, $v,w\in \frak{g}$. Since $v,w$ are themselves linear combinations of the $x_i$, the fact follows.) By definition of $\X$ being $Q$-rational (see \cite[Definition 2.4]{GTOrb}), each $[x_i,x_j]$ is a $Q$-rational combination of the $x_i$, so these rational combinations span $\frak{g}_{(3)}$, and so we can find a basis of such combinations for $\frak{g}_{(3)}$, whence $G_{(3)}$ is indeed $Q$-rational relative to $\X$. Next, note that since $G_\bullet':=(G_{(0)},G_{(1)},G_{(2)},G_{(3)},\{\id_G\},\ldots)$ is a filtration on $G$, we may apply \cite[Proposition A.10]{GTOrb} with $G'=G$ and the filtrations $G_\bullet, G_\bullet'$, and thus obtain the claimed basis $\X'$.\\
\indent Now we use $\X'$ to define a 2-step nilpotent Lie subgroup $H$ of $G$ and an abelian Lie subgroup $V$, which will yield our decomposition.
Letting $m_i$ denote the dimension of $G_{(i)}$, for $i\in [3]$ (and $m_0=\dim(G)$), we define first
\begin{eqnarray*}
H & = & e^{\R\cdot x_1}\, \cdots \, e^{\R\cdot x_{m_0-m_2}}\,\cdot\, e^{\R\cdot x_{m_0-m_3+1}}\,\cdots\, e^{\R\cdot x_{m_0}}\\
& = & e^{\R\cdot x_1}\cdots e^{\R\cdot x_{m_0-m_2}} \;\cdot [G,G].
\end{eqnarray*}
Thus, letting $\pi$ denote the quotient map $G\to G/[G,G]$, we have that $H$ is the preimage under $\pi$ of the $\R$-span of the vectors $\pi(e^{ x_1}),\ldots,\pi(e^{x_{m_0-m_2}})$. In particular $H$ is a normal Lie subgroup of $G$ (see for instance \cite[Theorem 3.5, p. 18]{O&V}), connected and simply-connected.\\
\indent Next, we define
\begin{eqnarray*}
V & = & e^{\R\cdot x_{m_0-m_2+1}}\,\cdots\, e^{\R\cdot x_{m_0-m_3}}.
\end{eqnarray*}
This is an abelian Lie group (connected and simply-connected), indeed it is just the vector subspace of $G_{(2)}$ spanned by the elements $e^{x_j}$, $j\in (m_0-m_2, m_0-m_3]$.\\
\indent These normal Lie subgroups $H,V$ of $G$ satisfy $H\cap V=\{\id\}$ and $H\cdot V= G$. In fact, we have an isomorphism of Lie groups $\theta: G\to H\times V$ defined by $\theta(g) = (h,v)$ where $g = e^{r_1 x_1}\cdots e^{r_{m_0} x_{m_0}}$, $h= e^{r_1 x_1}\cdots e^{r_{m_0-m_2} x_{m_0-m_2}}\, e^{r_{m_0-m_3+1} x_{m_0-m_3+1}}\cdots e^{r_{m_0} x_{m_0}}$, and $v= e^{r_{m_0-m_2+1} x_{m_0-m_2+1}}\cdots e^{r_{m_0-m_3} x_{m_0-m_3}}$.\\
\indent We claim that $\theta$ is an isomorphism of filtered Lie groups if we endow $H\times V$ with the filtration $H_{\bullet(1)}\times V_{\bullet(2)}$. To prove this we just need to show that $\theta(G_{(2)})=[H,H]\times V$. Since $G_{(2)}= e^{\R\cdot x_{m_0-m_2+1}}\,\cdots \,e^{\R\cdot x_{m_0}}$, by definition of $V$  we have that $\theta(G_{(2)}) = [G,G]\times V$. We also have $[G,G]=[H,H]$.  Indeed, on one hand we have $G\subset H\cdot G_{(2)}$, and on the other hand since $G$ is 2-step nilpotent we have that $G_{(2)}$ lies in the center of $G$ and the commutator map satisfies $[g_1g_2,g_3]=[g_1,g_3]\cdot [g_2,g_3]$, whence $[G,G]\subset [H\cdot G_{(2)},H\cdot G_{(2)}] \subset [H,H]$. This proves our claim.\\
\indent Note that by construction $H$ and $V$ are rational subgroups of $G$ and so by \cite[Theorem 5.1.11]{C&G} we have that the subgroups $\Gamma_H:=\Gamma\cap H$ and $\Gamma_V:=\Gamma\cap V$ are lattices in $H,V$ respectively. Moreover, the following are then Mal'cev bases for $(H/\Gamma_H,H_{\bullet(1)})$ and $(V/\Gamma_V,V_{\bullet(2)})$ respectively:
\[
\X_H= \{x_1,\ldots,x_{m_0-m_2},x_{m_0-m_3+1},\ldots, x_{m_0}\},\hspace{0.5 cm} \X_V= \{x_{m_0-m_2+1},\ldots,x_{m_0-m_3}\},
\]
where $\X_H$ is $Q^{O(1)}$ rational (and $\X_V$ is $0$-rational).  
Letting $\frak{h},\frak{v}$ denote the Lie algebras of $H,V$, the Lie algebra of $H\times V$ is $\frak{h}\oplus \frak{v}$, and then the Mal'cev basis $\X_{H\times V}$ for
\[
\big((H\times V)/(\Gamma_H\times \Gamma_V), H_{\bullet(1)}\times V_{\bullet(2)}\big)
\]
satisfying condition $(ii')$ in Definition \ref{def:BNilIso} is
\[
\{(x_1,0),\ldots,(x_{m_0-m_2}, 0), (0, x_{m_0-m_2+1}),\ldots,(0, x_{m_0-m_3}), (x_{m_0-m_3+1}, 0),\ldots, (x_{m_0}, 0)\}.
\]
Since $(V/\Gamma_V,V_{\bullet(2)},\X_V)$ is isomorphic to $(\T^m,\R^m_{\bullet(2)})$ with the standard basis, the proof is complete.
\end{proof}
\noindent Throughout the sequel, the Mal'cev basis on a torus $\T^m$ is by default the standard basis on $\R^m$, so we shall not specify the basis on filtered tori from now on.
\begin{remark}\label{rem:UseOfStruct}
The regularity lemma for the $U^3(\Zmod{p})$ norm (stated for instance in \cite[Theorem 5.1]{CS3}) can be refined using Proposition \ref{prop:prodecomp}, by replacing the unspecified nilmanifold $\ns$ of degree 2 in that result by a product nilmanifold of the form $\nss_1\times \nss_2$ given by the proposition. Note that in doing so it is important to control the complexity (or quantitative rationality)  of the basis on $\nss_1\times \nss_2$ only in terms of the complexity of $\ns$, so that the metric structures, governing Lipschitz constants etc., can be related in a way that depends only on the complexity of $\ns$. The unspecified isomorphism in Proposition \ref{prop:prodecompsoft} does not a priori enable such a control, but the one in Proposition \ref{prop:prodecomp} enables this easily. Let us state this refined regularity result.

\begin{theorem}\label{thm:periodic_regularity}
Let $s$ be a positive integer, let $\epsilon > 0$, and let $\mathcal{F} : \R_{>0} \to \R_{>0}$ be a growth function. Then there is a real number $M = O_{s,\epsilon,\mathcal{F}}(1)$ such that for any prime number $p\geq N_0(s, \epsilon, \mathcal{F})$ and any function $f : \Zmod{p} \to [0,1]$ there is a decomposition
\[ f = f_\nil + f_\sml + f_\unf \]
with the following properties:
\begin{enumerate}
\item $f_\nil: \Zmod{p} \to [0,1]$ is a $p$-periodic, $\mathcal{F}(M)$-irrational nilsequence of degree at most $s$ and complexity at most $M$, with underlying filtered nilmanifold of the form $\nss_1\times \nss_2$ where $\nss_1$ is a 2-step nilmanifold equipped with the lower central series, and $\nss_2$ is a torus equipped with the maximal degree-2 filtration.
\item $f_\sml : \Zmod{p} \to [-1,1]$ satisfies $\norm{f_\sml}_2 \leq \epsilon$.
\item $f_\unf: \Zmod{p} \to [-1,1]$ satisfies $\norm{f_\unf}_{U^{s+1}} \leq 1/\mathcal{F}(M)$.
\item $f_\nil + f_\sml$ takes values in $[0,1]$. 
\end{enumerate}
\end{theorem}
\end{remark}

We can now reduce the proof of Theorem \ref{thm:maincasereduc} as follows.

\begin{proposition}\label{prop:ReducToTori}
Suppose that $(\T,\R_{\bullet(1)})$ models every 2-step nilmanifold with lower central series, and that $(\T,\R_{\bullet(2)})$ models every torus with maximal degree-2 filtration. Then Theorem \ref{thm:maincasereduc} holds.
\end{proposition}
\noindent This can be proved using Proposition \ref{prop:prodecompsoft}, but to give a detailed proof it is convenient to use the more precise Proposition \ref{prop:prodecomp}.
\begin{proof}
Let $\ns=(G/\Gamma,G_\bullet)$ be a filtered nilmanifold of degree at most 2, fix any basis $\X$ on $\ns$, and let $\delta>0$. We apply Proposition \ref{prop:prodecomp}, and let $\nss_1=(H/\Gamma_H,H_{\bullet(1)},\X_H)$ and $\nss_2=(\T^m,\R^m_{\bullet(2)})$ be the resulting nilmanifolds and $\theta$ be the resulting isomorphism $H\times\R^m\to G$ satisfying conditions (i) and (ii') from Definition \ref{def:BNilIso}. By assumption, for every $\delta_1,\delta_2>0$ there exists a $\delta_1$-balanced map $\phi_{g_1}: (\T,\R_{\bullet(1)})\to (H/\Gamma_H,H_{\bullet(1)},\X_H)$ and a $\delta_2$-balanced map $\phi_{g_2}: (\T,\R_{\bullet(2)})\to (\T^m,\R^m_{\bullet(2)})$ (recall the notation $\phi_g$ from the paragraph before Definition \ref{def:multib}). Letting $g_1\times g_2$ denote the product map $(r_1,r_2)\mapsto (g_1(r_1),g_2(r_2))$, we now let $g:\R^2\to G$ be the composition $\theta\circ (g_1\times g_2)$. Noting that $\theta$ is a $(\Gamma_H\times \Z^m,\Gamma)$-consistent polynomial map and that $g_1\times g_2$ is a $(\Z^2,\Gamma_H\times \Z^m)$-consistent polynomial map, we have that $g$ is a $(\Z^2,\Gamma)$-consistent map in $\poly_0(H_{\bullet(1)}\times\R^m_{\bullet(2)},G_\bullet)$. It now suffices to show that $\delta_1,\delta_2$ can be chosen in terms of $\delta,\ns$ so that $\phi_g$ is $\delta$-balanced.\\
\indent Let $\overline{\theta}: \nss_1\times \nss_2\to \ns$ denote the homeomorphism induced by $\theta$, and note that $\overline{\theta}$ is a bilipschitz map with constant $O_Q(1)$, where $Q$ is the complexity bound on $\ns$ (this can be checked using results such as \cite[Lemma A.17]{GTOrb}). Given a system $\Lf:\Z^D\to \Z^t$ of size at most $1/\delta$ and a function $F_0:\ns^\Lf\to \C$ with $\|F_0\|_{\Lip}\leq 1$, it follows that $F:=F_0\circ (\overline{\theta}^{\,t})$ is an $O_{Q,\delta}(1)$-Lipschitz function on $(\nss_1\times \nss_2)^\Lf$. Note that $(\nss_1\times \nss_2)^\Lf$ is isomorphic to $\nss_1^\Lf\times \nss_2^\Lf$. Viewing $F$ as a function on the latter product space, we may approximate $F$ within $\delta/4$ in the supremum norm by a finite sum $\sum_{i\in [M]}F_i$, where for each $i$ we have $F_i:(y_1,y_2)\mapsto F_{i,1}(y_1)F_{i,2}(y_2)$ where $F_{i,j}$ is $O_{Q,\delta}(1)$-Lipschitz on $\nss_j^\Lf$ for $j=1,2$. (To see that such an approximation exists, one can first use that for any $C>0$ the set of functions $f$ on $\nss_1^\Lf\times \nss_2^\Lf$ with $\|f\|_{\Lip}\leq C$ is totally bounded in the supremum norm; this follows from the Arzel\`a-Ascoli theorem. In particular, for $C=O_{Q,\delta}(1)\geq \|F\|_{\Lip}$, there is a finite $\delta/8$-net for the set of $f$ on $\nss_1^\Lf\times \nss_2^\Lf$ with $\|f\|_{\Lip}\leq C$. By the Stone-Weierstrass theorem, each function in this net is within $\delta/8$ in the supremum norm from a function $F_i$ of the claimed form; hence $F$ is within $\delta/4$ of such a function.) Now if for $j=1,2$ we have $\big|\int_{(\T,\R_{\bullet(j)})^\Lf} F_{i,j}\circ \phi_{g_j}^t -\int_{\nss_j^\Lf} F_{i,j}\big|\leq \delta_j$, then $\big|\int_{\ns_2^\Lf} F_i\circ (\phi_{g_1}\times \phi_{g_2})^t- \int_{(\nss_1\times\nss_2)^\Lf} F_i\big|\leq \delta_1 \|F_{i,2}\|_\infty+\delta_2 \|F_{i,1}\|_\infty$. It follows that we can choose $\delta_1,\delta_2$ in terms of $\delta,Q$ to obtain that $\big|\int_{\ns_2^\Lf}F_0\circ \phi_g^t-\int_{\ns^\Lf}F_0\big|\leq \delta$, as required.
\end{proof}

\section{Balanced maps from the circle to 2-step nilmanifolds with lower central series}\label{sec:LC-maps}

\noindent Our aim here is to establish the following result.

\begin{proposition}\label{prop:circ-to-Y1}
Let $\nss=(H/\Gamma,H_{\bullet(1)})$ be a 2-step nilmanifold with lower central series, and let $\ns_1=(\T,\R_{\bullet(1)})$. Then $\ns_1$ models $\nss$. 
\end{proposition}
\noindent From previous work we already have the discrete version of this proposition in which $\ns_1$ is replaced with $\Zmod{p}$; more precisely, we have the following more general result.
\begin{proposition}[Existence of a balanced periodic polynomial sequence]\label{prop:per-bal}\hfill\\
Let $\delta>0$, and let $(G/\Gamma,G_\bullet,\X)$ be a filtered based nilmanifold of complexity at most $m$. Then there exists $C=C(m,\delta)>0$ such that for every prime $p \geq C$ there exists $g \in \poly_0(\Z,G_\bullet)$ that is $(p\Z,\Gamma)$-consistent and such that the map $\phi_g:\Zmod{p}\to G/\Gamma$, $n\mapsto g(n)\Gamma$ is $\delta$-balanced. Moreover, if $G_\bullet=G_{\bullet(1)}$, then $g$ can be taken to be linear.
\end{proposition}
\noindent A polynomial $g\in \poly_0(\Z,G_\bullet)$ is said to be \emph{linear} if it is of the form $g(n)=g_1^n$, for some $g_1\in G$.\\
\indent Proposition \ref{prop:per-bal} is essentially \cite[Proposition 6.1]{CS3}. To obtain the additional linearity claim for $g$ above, the main fact used is that for the lower-central series $G_{\bullet(1)}$ there are no non-trivial $i$\emph{-th level characters}  with $i>1$ (see \cite[Definition 4.2]{CS3}). It is then a simple task to find $g_1$ so that $g(n)=g_1^n$ has the desired properties (see \cite[\S 6]{CS3}).\\
\indent From the `discrete time'  result Proposition \ref{prop:per-bal}, we shall deduce the `continuous time' result Proposition \ref{prop:circ-to-Y1}. This can be done using the quotient-integral formula. Let us illustrate this in the simplest setting, namely the case of equidistribution just for 1-parameter orbits: if $g_1^n\Gamma$ is equidistributed in $G/\Gamma$, then for every $g_0\in G$ the sequence $g_0g_1^n\Gamma$ is also equidistributed in $G/\Gamma$. In particular, for every Lipshitz function $F$ on $G/\Gamma$ with $\|F\|_{\Lip(\X)}\leq 1$ and every $r\in [0,1/p)$ we have\footnote{For $a,b\in \C$ we write $a\approx_\epsilon b$ if and only if $|a-b|\leq \epsilon$.}
\[
\E_{n\in \Zmod{p}} F(g_1^{rp+n}\Gamma)\approx_\epsilon \int_{\ns} F \ud\mu_{\ns}.
\]
The periodicity and linearity of $g$ imply that $\theta\mapsto g_1^{\theta p}\Gamma$ is a well-defined (continuous) map $\T\to G/\Gamma$. Then, by the quotient integral formula \cite[Theorem 1.5.2]{principlesHA}, we have
\begin{eqnarray*}
\int_{\T} F(g_1^{\theta p} \Gamma) \ud\mu_\T(\theta) & = & \int_{r\in [0,1/p)} \E_{n\in \Zmod{p}} F(g_1^{(r+n/p)p}\Gamma) \;\;p\, \ud\mu_\T(r)\\
& \approx_\epsilon & \int_{r\in [0,1/p)} \int_{\ns} F \ud\mu_{\ns} \;\;p\, \ud\mu_\T(r) = \int_{\ns} F \ud\mu_{\ns}.
\end{eqnarray*}
Let us now prove the general case.

\begin{proof}[Proof of Proposition \ref{prop:circ-to-Y1}]
Fix a basis on $\nss=(H/\Gamma,H_{\bullet(1)})$ and let $\delta>0$. Our task is to produce a $\delta$-balanced map $\ns_1\to \nss$. Let $g:\Z\to H$ be the polynomial map given by Proposition \ref{prop:per-bal}, such that the induced map $\phi_g:\Zmod{p}\to H/\Gamma$, $n\mapsto g(n)\Gamma$ is $\delta$-balanced. From periodicity we have that $g(n)=\gamma^{n/p}$ for some $\gamma\in \Gamma$. Let $\Lf:\Z^D\to\Z^t$ be a system of integer linear forms of size at most $1/\delta$. Then $g^t(\Lf(n))\Gamma^\Lf$ is $\delta$-equidistributed in $\nss^\Lf$. Let $\phi:\T\to H/\Gamma$ be the circle flow interpolating the map $\phi_g$, that is $\phi:\T \to H/\Gamma$, $\theta\mapsto \gamma^{\theta}\Gamma$. Our main claim is that
\[
\Big|\int_{\T^D} F^{\otimes^t}\circ\phi^t(\Lf(x)) \ud\mu_{\T^D}(x) - \int_{\nss^\Lf}F^{\otimes^t}\Big|\leq \delta.
\]
Using the notation $\gamma^v=\begin{psmallmatrix}\gamma^{v(1)}\\[-0.4em]\vdots \\[0.1em] \gamma^{v(t)}\end{psmallmatrix}\in G^t$ for $v\in \R^t$ (see Definition \ref{def:LeibGp}), the integral on the left side here is written $\int_{\T^D} F^{\otimes^t}\big(\gamma^{\Lf(x)}\Gamma^t\big) \ud\mu_{\T^D}(x)$. By the quotient integral formula, this equals
\[
\int_{[0,1/p)^D}\left(\E_{n\in \Zmod{p}^D} F^{\otimes^t}\big(\gamma^{\Lf(r)+\Lf(\frac{n}{p})}\Gamma^t\big)\right)\;p^D\cdot \mu_{\T^D}(r)
\]
For each $r\in [0,1/p)^D$, the $\Zmod{p}^D$-orbit $(\gamma^{\Lf(r)}\cdot \gamma^{\Lf(\frac{n}{p})})\Gamma^t)$ is still $\delta$-equidistributed in $G^\Lf/\Gamma^\Lf$, since the equidistribution property is not affected by multiplying by the constant $\gamma^{\Lf(r)}\in G^\Lf$. Thus for each such $r$ we have
\[
\Big| \E_{n\in \Zmod{p}^D} F^{\otimes^t}\big(\gamma^{\Lf(r)}\cdot \gamma^{\Lf(\frac{n}{p})}\Gamma^t\big) - \int_{\nss^\Lf} F^{\otimes^t} \Big|\leq \delta,
\]
and the result follows.
\end{proof}

\noindent Let us end this section by using Proposition \ref{prop:per-bal} to establish the easy case of Theorem \ref{thm:comp2trans}, as follows.

\begin{proposition}\label{prop:easycase}
For every $\epsilon >0$ there exists $C>0$ such that the following holds. For every measurable function $f : \ns_2 \to [0,1]$ and every prime $p\geq C$, there is a function $f' : \Zmod{p}\to [0,1]$ such that, for every system $\Lf$ of linear forms of size at most $1/\epsilon$ and complexity at most 2, we have $\abs{ S_{\Lf}(f : \ns_2) - S_\Lf(f' : \Zmod{p}) } \leq \epsilon$.
\end{proposition}

\begin{proof}
We first claim that $f$ can be assumed to be continuous with Lipschitz norm depending only on $\epsilon$, more precisely there exists $f_0$ with $\|f_0\|_{\Lip(\ns_2)}\leq C'(\epsilon)$ such that $|\Sol_\Lf(f:\ns_2)-\Sol_\Lf(f_0:\ns_2)|\leq \epsilon/2$ for every $\Lf$ of size at most $1/\epsilon$. We prove this with the following compactness argument.\\
\indent Let $\cL(\epsilon)$ denote the finite set  of all systems of forms $\Lf$ of size at most $1/\epsilon$. Let us say that a point $w\in [0,1]^{\cL(\epsilon)}$ is \emph{achieved} if there is a continuous function $f_w:\ns_2\to [0,1]$ such that $\Sol_{\Lf}(f:\ns_2)= w(\Lf)$ for every $\Lf\in \cL(\epsilon)$. Let $W$ denote the closure of the set of achieved points in $[0,1]^{\cL(\epsilon)}$ with respect to the $\ell^\infty$ norm. By compactness of $W$ there exists a finite $(\epsilon^2/4)$-net $W_\epsilon$ of achieved points. Each continuous function $f_w$ for $w\in W_\epsilon$ can be assumed to have finite Lipschitz norm, since the Lipschitz functions on $\ns_2$ are dense, relative to the supremum norm, in the set of continuous functions (by the Stone-Weierstrass theorem). Let $C'=\max\{\|f_w\|_{\Lip(\ns_2)}:w\in W_\epsilon\}$. Now, given the measurable function $f$, first by Lusin's theorem there is a continuous function $f_c:\ns_2\to [0,1]$ such that $\|f-f_c\|_{L^1(\ns_2)}\leq \epsilon^2/4$, and we therefore have $|\Sol_\Lf(f:\ns_2)-\Sol_\Lf(f_c:\ns_2)|\leq \epsilon/4$ for all $\Lf\in \cL(\epsilon)$ (using the $L^1$-continuity described in Remark \ref{rem:L1cont}). Then for some $w\in W_\epsilon$ we have $\|f_c-f_w\|_\infty\leq\epsilon^2/4$, and it follows that $|\Sol_\Lf(f:\ns_2)-\Sol_\Lf(f_w:\ns_2)|\leq \epsilon/2$ for all $\Lf\in \cL(\epsilon)$. Relabelling $f_w$ as $f_0$, our claim follows.\\
\indent Now, given that $\|f_0\|_{\Lip(\ns_2)}\leq C'(\epsilon)$, we may apply Proposition \ref{prop:per-bal} with $\ns_2$ and $\delta$ sufficiently small depending only on $\epsilon$, so that the resulting map $\phi_g:\Zmod{p}\to \ns_2$ satisfies $|\Sol_\Lf(f_0:\ns_2)-\Sol_\Lf(f_0\circ \phi_g:\Zmod{p})|\leq \epsilon/2$ for all $\Lf\in \cL(\epsilon)$. We can then take $f'= f_0\circ\phi_g$. 
\end{proof}

\section{Balanced maps between tori of degree 2}\label{sec:QT-maps}

\noindent For each positive integer $d$ let us denote by $\nss(d)$ the $d$-dimensional torus with maximal degree 2 filtration, $\nss(d)=(\T^d, \R^d_{\bullet(2)})$. The main result of this section is the following.
\begin{proposition}\label{prop:2circ-to-2torus}
For every positive integer $d$ we have that $\nss(1)$ models $\nss(d)$.
\end{proposition}

\noindent For each positive integer $k$, we define the following continuous homomorphism:
\[
\phi_k: \T \to \T^d,\;\; x \mapsto (x,k x,k^2 x,\ldots, k^{d-1} x).
\]
We shall prove the proposition by showing that for every $\delta>0$, for $k$ sufficiently large the map $\phi_k$ is a $\delta$-balanced map $\nss(1) \to \nss(d)$. As we shall eventually see in Lemma \ref{lem:nontrivchar} below, the main property enabling this is that, for every $j\in [d-1]$, for every $r_0,r_1,\dots,r_{j-1}\in \Z$ and $r_j\in\Z\setminus\{0\}$, we have $r_0+r_1 k+\cdots+r_j k^j\neq 0$ for $k$ sufficiently large (there are of course other choices of $\phi_k$ with this property).\\
\indent Given a system of forms $\Lf:\Z^D\to \Z^t$ of size at most $1/\delta$, viewing $\Lf$ as a matrix (cf. Definition \ref{def:LeibGp}), for each $i\in [D]$ let $u_i$ denote the $i$-th column $\Lf(e_i)\in \Z^t$ of  $\Lf$ (where $e_1,\dots,e_D$ is the standard basis of $\R^D$). Recall from Definition \ref{def:LeibGp} that $\Lf^{[2]}$ is the subgroup of $\Z^t$ generated by the collection of vectors consisting of the $u_i$, the products $u_i u_j$, and the vector binomial coefficients $\binom{u_i}{2}$. Let $v_1,\ldots,v_m\in \Z^t$ be a set of generators for $\Lf^{[2]}$. We then have $\Lf^{[2]} =  \Z\,v_1+ \cdots + \Z\,v_m$ and we can express the Leibman nilmanifolds for $\Lf$ on $\nss(1),\nss(d)$ as follows:
\begin{eqnarray*}
\nss(1)^\Lf & = & \T\,v_1+ \cdots + \T\,v_m \leq \T^t,\\
\nss(d)^\Lf & = & \T^d\,v_1+ \cdots + \T^d\,v_m \leq (\T^d)^t,
\end{eqnarray*}
where $\T^d\,v=\Big\{ \begin{psmallmatrix}\theta v(1)\\[-0.4em]\vdots \\[0.1em] \theta v(t)\end{psmallmatrix}:\theta \in \T^d\Big\}$. 

Now $\phi_k^t$ maps $\nss(1)^\Lf$ into $\nss(d)^\Lf$ (as can be checked directly or by Proposition \ref{prop:polykeepLeib}), and we want to show that this map is $\delta$-equidistributed if $k$ is large enough.

By Definition \ref{def:equid}, we have to show that for any function $F:\nss(d)^\Lf\to \C$ with $\|F\|_{\Lip(\T^{dt})}\leq 1$, we have
$\left|\int_{\nss(d)^\Lf} F\; \ud\mu_{\nss(d)^\Lf} - \int_{\nss(1)^\Lf} F\circ \phi_k^t\;\ud\mu_{\nss(1)^\Lf}\right|\leq \delta$. 
We use the following result on Fourier approximations of Lipschitz functions \cite[Lemma A.9]{GTMob}.
\begin{lemma}\label{fourier-lip} 
Let $\T^r$ be the standard $r$-dimensional torus, with metric induced by the $\ell^\infty$ norm $\| (x_1,\ldots,x_r) \|_{\T^r} := \sup_{1 \leq j \leq r} \| x_j \|_\T$. Let $X$ be a subset of $\T^r$, and let $f: X \to \C$ be a Lipschitz function.  Then for every positive integer $N$ there exist $J = O_{r}(N^r)$, $c_1,\ldots,c_J = O(\|f\|_{\infty})$, and $m_1,\ldots,m_J \in \Z^r$ such that for all $x \in X$ we have
\[
 f(x) = \sum_{j=1}^J c_j\; e(m_j \cdot x) + O_r\Big(\frac{\| f \|_{\Lip(\T^r)} \log N}{N}\Big).
 \]
Furthermore, the values of $m_1,\ldots,m_J$ depend on $r$, $N$ but are otherwise independent of $f$ or $X$.
\end{lemma}

\noindent We apply this in our situation, with $X=\nss(d)^\Lf$, $r=td$, $f=F$. Choosing $N$ sufficiently large, denoting the obtained characters  $x\mapsto e(m_j\cdot x)$ by $e_{m_j}$, and letting $F_0=\sum_{j=1}^J c_j\, e_{m_j}$,  we have by the lemma that $\|F - F_0\|_\infty\leq \delta/4$. It therefore suffices to show that
\begin{equation}\label{eq:FourierSumEquid}
\left|\;\int_{\nss(d)^\Lf} F_0\; \ud\mu_{\nss(d)^\Lf}\;-\;\int_{\nss(1)^\Lf} F_0\circ \phi_k^t\; \ud\mu_{\nss(1)^\Lf}\;\right| \leq \delta/2.
\end{equation}
We shall in fact prove this with upper bound equal to 0, by showing that for sufficiently large $k$ we have for every $j\in [J]$ that
\begin{equation}\label{eq:chareq}
\int_{\nss(d)^\Lf} e_{m_j}(y) \ud\mu_{\nss(d)^\Lf}(y) = \int_{\nss(1)^\Lf} e_{m_j}(x,kx,\ldots,k^{d-1}x) \ud\mu_{\nss(1)^\Lf}(x). 
\end{equation}
Since \eqref{eq:chareq} clearly holds when $e_{m_j}$ restricted to $\nss(d)^\Lf$ is the principal character, and otherwise we have $\int_{\nss(d)^\Lf} e_{m_j}(y) \ud\mu_{\nss(d)^\Lf}(y)=0$, it will suffice  to prove the following result.

\begin{lemma}\label{lem:nontrivchar}
Let $\chi$ be a character on $(\T^d)^t$ such that the restriction $\chi|_{\nss(d)^\Lf}$ is not the principal character on $\nss(d)^\Lf$. Then  for every sufficiently large positive integer $k$, we have
\[
\int_{\nss(1)^\Lf} \chi(x,kx,\ldots,k^{d-1}x) \ud\mu_{\nss(1)^\Lf}(x) =0.
\]
\end{lemma}
\noindent Indeed, applying this above for each $j\in [J]$, we deduce  \eqref{eq:FourierSumEquid} for some $k\in \N$, as desired.
\begin{proof}
Since $(\T^d)^t\cong \bigoplus_{i\in [d]} \T^t$, there are characters $\chi_1,\dots,\chi_d$ on $\T^t$ such that
\[
\chi(x,kx,\ldots,k^{d-1}x)=\chi_1(x)\chi_2(kx)\cdots \chi_d(k^{d-1}x).
\]
Moreover, since $\nss(d)^\Lf\cong \bigoplus_{i\in [d]}  \nss(1)^\Lf$, at least one $\chi_j$ restricts to a non-principal character on $\nss(1)^\Lf$. Let $j\in [d]$ be the greatest index such that this is the case. Then for $k$ large enough, the character on $\nss(1)^{\Lf}$ sending $x$ to $\chi(x)=\chi_1(x)\chi_2(kx)\cdots \chi_j(k^{j-1}x)$ is non-trivial, since its frequency is a non-zero vector. The result follows.
\end{proof}

\noindent With Propositions \ref{prop:2circ-to-2torus}, \ref{prop:circ-to-Y1}, and \ref{prop:ReducToTori}, the proof of Theorem \ref{thm:maincasereduc} is now  complete, and this implies Proposition \ref{prop:maincase} (by Proposition \ref{prop:mcr-implies-ms}). This together with Proposition \ref{prop:easycase} then gives us finally the transference result Theorem \ref{thm:comp2trans}, and Theorem \ref{thm:main} follows.

\section{On modelling nilmanifolds of higher degree -- a counterexample}\label{sec:h3}

\noindent The main result of this section concerns possible generalizations of Theorem \ref{thm:main} for systems of higher complexity. For each positive integer $s$ let $\ns_s$ denote the following filtered torus of degree $s$:
\[
\ns_s=(\T^s,\R_{\bullet(1)}\times\cdots\times \R_{\bullet(s)}).
\]
One may believe at first that the natural generalization of Theorem \ref{thm:main} for complexity $s>2$ should consist in replacing $\ns_2$ with $\ns_s$. This generalization would hold if it were true that $\ns_s$ models every filtered nilmanifold $\ns$ of finite complexity and degree at most $s$. However, this claim fails already for $s=3$, as we show in this section. More precisely, we prove that $\ns_3$ fails to model the nilmanifold of degree $3$ defined as follows.

Let $H$ be the Heisenberg group $\begin{psmallmatrix} 1 & \R & \R\\[0.1em]  & 1 & \R \\[0.1em]  &  & 1 \end{psmallmatrix}:=\Big\{ \begin{psmallmatrix} 1 & x_1 & x_3\\[0.1em]  & 1 & x_2 \\[0.1em]  &  & 1 \end{psmallmatrix}: x_i \in \R\Big\}$, let $\Gamma=\begin{psmallmatrix} 1 & \Z & \Z\\[0.1em]  & 1 & \Z \\[0.1em]  &  & 1 \end{psmallmatrix}$, and let $H_\bullet$ denote the degree-3 filtration on $H$ with $H_{(2)}=\begin{psmallmatrix} 1 & 0 & \R\\[0.1em]  & 1 & \R \\[0.1em]  &  & 1 \end{psmallmatrix}$ and $H_{(3)}=\begin{psmallmatrix} 1 & 0 & \R\\[0.1em]  & 1 & 0 \\[0.1em]  &  & 1 \end{psmallmatrix}$. The fact that this is indeed a filtration is checked directly; in particular since $H_{(2)}$ is abelian (isomorphic to $\R^2$), we have $[H_{(2)},H_{(2)}]\subset H_{(4)}=\{\id_H\}$. Let $\nh_3$ denote the filtered nilmanifold $(H/\Gamma,H_\bullet)$. Let $\tilde\pi$  denote the quotient homomorphism $H\to H/H_{(3)}\cong \R^2$, and let $\pi$ denote the induced projection $H/\Gamma \to \T^2$, defined by $\pi(x\Gamma)=\tilde\pi(x)+\Z^2$. Note that $\tilde\pi$ is a $(\Gamma,\Z^2)$-consistent map in $\poly_0\big( H_\bullet,\R_{\bullet(1)}\times\R_{\bullet(2)}\big)$.

The main result of this section is the following.

\begin{proposition}\label{prop:countercomp3}
The filtered torus $\ns_3$ does not model the filtered nilmanifold $\nh_3$.
\end{proposition}

\noindent Before going into the details, let us convey the idea of the proof.\\
\indent Suppose for a contradiction that, for every fixed $\epsilon>0$, there existed an $\epsilon$-balanced map $\phi_g:\ns_3\to \nh_3$, induced by a $(\Z^3,\Gamma)$-consistent continuous map
\[
g \in \poly_0(\R_{\bullet(1)}\times \R_{\bullet(2)}\times \R_{\bullet(3)}, H_\bullet).
\]
The idea is that, if $\epsilon$ is sufficiently small, then the assumed equidistribution property would imply that $\phi_g$ yields a type of continuous `cross section' $\T^2\to H/\Gamma$ which cannot exist. We make this idea precise in two main steps. The first step consists in showing that composing $\phi_g$ with the projection $\pi$ induces a polynomial map $\beta:\ns_2\to\ns_2$ which, if $\epsilon$ is sufficiently small, must be a certain type of surjective homomorphism (this is the combination of lemmas \ref{lem:beta} and \ref{lem:eps}). This gives us precise information on the form that the polynomial $g$ itself must have. In the second step we use this information to show, essentially, that a restriction of such a map $\phi_g$ to a cross section of $\T^2$ in $\T^3$ already yields a contradiction, because the assumed form of $g$ in fact does not allow it to be a consistent polynomial (this is made precise in Lemma \ref{lem:calc}).

Let us now turn to the details. Let $\alpha$ denote the projection $\T^3\to\T^2$, $x\mapsto (x_1,x_2)$ and note that $\alpha$ is a polynomial map $\ns_3\to \ns_2$. We first show that the map $\pi\circ\phi_g$ factors through $\alpha$.

\begin{lemma}\label{lem:beta}
There is a polynomial map $\beta:\ns_2\to \ns_2$ such that $\pi\circ\phi_g = \beta\circ\alpha$.
\end{lemma}
\noindent We use the following fact, which is a straightforward consequence of the definitions.
\begin{lemma}\label{lem:consisquot}
If $g\in \poly(G_\bullet,G_\bullet')$ is $(\Gamma,\Gamma')$-consistent  and $\Gamma,\Gamma'$ are normal subgroups of $G,G'$ respectively, then the induced map $\phi_g$ is polynomial from $G/\Gamma$ to $G'/\Gamma'$ with respect to the quotient filtrations on these groups.\footnote{The $i$-th group in the quotient filtration of $G_\bullet$ by $\Gamma$ is $(G_{(i)}\cdot \Gamma)/\Gamma$.}
\end{lemma}
\begin{proof}[Proof of Lemma \ref{lem:beta}]
For any two fixed points $a,b\in\R$, let
\[
f_{a,b}\colon \R\to \R^2,\quad x\mapsto \tilde\pi(g(a,b,x)).
\]
Since $x\mapsto g(a,b,x)$ and $\tilde \pi$ are polynomial, we have $f_{a,b}\in \poly(\R_{\bullet(3)},\R_{\bullet(1)}\times \R_{\bullet(2)})$. From the $(\Z^3,\Gamma)$-consistency of $g$, we deduce that $f_{a,b}$ is $(\Z,\Z^2)$-consistent. By Lemma \ref{lem:consisquot}, the induced map $\phi=\phi_{f_{a,b}}:\T\to \T^2$ is itself a polynomial map from $(\T,\R_{\bullet(3)})$ to $\ns_2$. We claim that this map is constant. To prove this we show that the compositions of $\phi$ with each of the two coordinate projections on $\ns_2$ are constant.\\
\indent Indeed, consider first the composition of $\phi$ with the projection to the first coordinate $\ns_2\to \ns_1$. This is a continuous polynomial map $\phi_1: (\T,\R_{\bullet(3)})\to \ns_1$, which implies that $\phi_1$ is a polynomial map $\T\to\T$ of degree $\leq 1$, by Lemma \ref{lem:ctspolylin}. Moreover, as a polynomial map $\phi_1$ must also conserve $3$-dimensional cubes; this follows from Proposition \ref{prop:polykeepLeib} applied to the system of forms $\Lf_{U^3}$ corresponding to 3-cubes (recall from \eqref{eq:U3} that $\Lf_{U^3}(n_1,n_2,n_3,n_4)=(n_1+v\cdot (n_2,n_3,n_4))_{v\in \{0,1\}^3}$). Now computing the Leibman group on $(\T,\R_{\bullet(3)})$ for this system, we find that $(\T,\R_{\bullet(3)})^{\Lf_{U^3}}\cong \T^{\{0,1\}^3}$. On the other hand, we have that $\phi_1\circ c$ is a 3-cube over $\ns_1$ (i.e. lies in $(\T,\R_{\bullet(1)})^{\Lf_{U^3}}$) if and only if it is a function of $v\in \{0,1\}^3$ of the form $\phi_1\circ c (v)=x+v\cdot h$, for some $x\in \T,h=(h_1,h_2,h_3)\in \T^3$. If $\phi_1$ were non-constant on $\T$, then the function $\T^{\{0,1\}^3}\to \T^{\{0,1\}^3}$, $c\mapsto  \phi_1\circ c$ would have to map the 8-dimensional group $(\T,\R_{\bullet(3)})^{\Lf_{U^3}}$ to the 4-dimensional group $\ns_1^{\Lf_{U^3}}$, which is impossible. Hence $\phi_1$ must be constant.\\
\indent By a similar argument, the composition of $\phi$ with projection to the second coordinate on $\ns_2$ is shown to be constant, using the fact that the torus  $(\T,\R_{\bullet(2)})^{\Lf_{U^3}}$ is 7-dimensional.\\
\indent Now we define $\beta:\ns_2\to \ns_2$, $(a+\Z,b+\Z)\mapsto \phi_{\tilde\pi\circ g}(a+\Z,b+\Z,x+\Z)$, for a fixed $x+\Z\in \T$. It follows from the fact that $\phi_{f_{a,b}}$ is constant that $\beta$ is independent of $x$, and we have $\pi\circ\phi_g = \beta\circ\alpha$.
\end{proof}

\begin{lemma}\label{lem:eps}
If $\epsilon>0$ is sufficiently small then $\beta(a,b)=(n_1a,n_2a+n_3b)$ where $n_1,n_2,n_3$ are integers and $n_1n_3\neq 0$.
\end{lemma}
\begin{proof}
First we claim that, as a polynomial map from $\ns_2$ to itself, $\beta$ must be of the form
\begin{equation}\label{eq:beta}
\beta(a,b)=(n_1a,n_2a+n_3b)
\end{equation}
for integers $n_1,n_2,n_3$. To see this, note first that composing $\beta$ with coordinate projections and using Lemma \ref{lem:ctspolylin}, we have that $\beta(a,b)=(n_1a+n_1'b,n_2a+n_3b)$ for integers $n_1,n_1',n_2,n_3$. Moreover, since for each fixed $a$ the map $b\mapsto n_1 a+n_1'b$ is by assumption polynomial from $(\T,\R_{\bullet(2)})$ to $(\T,\R_{\bullet(1)})$, we must have $n_1'=0$, by considering cubes in an argument similar to the one in the previous proof. Thus \eqref{eq:beta} holds.\\
\indent It remains to show that $n_1,n_3$ must both be non-zero.\\
\indent If $n_1$ were zero, then the map $\beta$ could not be $\epsilon$-equidistributed for sufficiently small $\epsilon$, and so $\phi_g$ could not be $\epsilon$-balanced.\\
\indent If $n_3$ were zero, then $\phi_g$ could still be $\epsilon$-equidistributed, but it could not be $\epsilon$-balanced for $\epsilon$ sufficiently small. Indeed, the map $\beta(a,b)$ would then depend only on $a$ and would thus induce the polynomial map $\beta':\ns_1\to \ns_2$, $a\mapsto (n_1,n_2)a$. However, then $\beta'$ cannot be $\epsilon$-balanced for $\epsilon$ sufficiently small, as can be seen by considering for instance the system $\Lf_{U^2}$ corresponding to 2-cubes (recall from \eqref{eq:U2} that $\Lf_{U^2}( n_1,n_2,n_3)=(n_1+v\cdot (n_2,n_3))_{v\in \{0,1\}^2}$). Indeed, for any cube $\q:\{0,1\}^2\to \T$ in the Leibman group $(\T,\R_{\bullet(1)})^{\Lf_{U^2}}$, on one hand the linearity of $\beta'$ implies that $\beta'\circ \q$ lies in $(\T^2,\R_{\bullet(1)}\times \R_{\bullet(1)})^{\Lf_{U^2}}$, this being the $6$-dimensional subgroup of $(\T^2)^{\{0,1\}^2}\cong \T^8$ consisting of all the $2$-cubes of degree 1, i.e. maps of the form $v\mapsto x+v_1h_1+v_2h_2$ for some $x,h_1,h_2\in \T^2$. On the other hand, the group $\ns_2^{\Lf_{U^2}}=(\T^2,\R_{\bullet(1)}\times \R_{\bullet(2)})^{\Lf_{U^2}}$ is the greater, 7-dimensional subgroup of $\T^8$ consisting of maps $v\mapsto (x+v_1h_1+v_2h_2, y_v)$, where $x,h_1,h_2,y_v$ are seven independent parameters in $\T$. Consequently, for $\epsilon$ sufficiently small $\beta'\circ \q$ cannot be $\epsilon$-equidistributed in $\ns_2^{\Lf_{U^2}}$ as $\q$ ranges in the group $\ns_1^{\Lf_{U^2}}$.
\end{proof}
\noindent Let us sum up the information on $g$ that we have gathered so far.\\
\indent The last two lemmas combined tell us that $g$ must be a polynomial map from $(\R^3,\R_{\bullet(1)}\times \R_{\bullet(2)}\times \R_{\bullet(3)})$ to $(H,H_\bullet)$ of the form $g(a,b,c)= \begin{psmallmatrix} 1 & n_1 a & r(a,b,c)\\[0.1em]  & 1 & n_2a+n_3b \\[0.1em]  &  & 1 \end{psmallmatrix}$, where $n_i$ are integers with $n_1,n_3$ non-zero, and where $r$ is a real-valued polynomial in the real variables $a,b,c$ such that $g$ is $(\Z^3,\Gamma)$-consistent.\\
\indent Letting $q(a,b)=r(a,b,0)$, we deduce that $h(a,b):=\begin{psmallmatrix} 1 & n_1 a & q(a,b)\\[0.1em]  & 1 & n_2a+n_3b \\[0.1em]  &  & 1 \end{psmallmatrix}$ is a $(\Z^2,\Gamma)$-consistent continuous polynomial map from $(\R^2,\R_{\bullet(1)}\times \R_{\bullet(2)})$ to $(H,H_\bullet)$. We shall now obtain a contradiction by examining the properties that $q$ must satisfy.

\begin{lemma}\label{lem:calc}
Let $q:\R^2\to \R$ be a polynomial such that $h(a,b):=\begin{psmallmatrix} 1 & n_1 a & q(a,b)\\[0.1em]  & 1 & n_2a+n_3b \\[0.1em]  &  & 1 \end{psmallmatrix}$ is a $(\Z^2,\Gamma)$-consistent polynomial map from $(\R^2,\R_{\bullet(1)}\times \R_{\bullet(2)})$ to $(H,H_\bullet)$. Then $n_1n_3=0$.
\end{lemma}

\begin{proof}
Given any real numbers $a,b,k_1,k_2$, we have that $h(a,b)^{-1}\,h(a+k_1,b+k_2)$ equals
\[
\begin{psmallmatrix} 1 &\hspace{0.5cm} -n_1 a &\hspace{0.5cm} -q(a,b)+n_1a(n_2a+n_3b)\\[0.1em]  &\hspace{0.5cm} 1 &\hspace{0.5cm} -n_2a-n_3b \\[0.1em]  &  & 1 \end{psmallmatrix} \begin{psmallmatrix} 1 &\hspace{0.5cm} n_1 (a+k_1) &\hspace{0.5cm} q(a+k_1,b+k_2)\\[0.1em]  & 1 &\hspace{0.5cm} n_2(a+k_1)+n_3(b+k_2) \\[0.1em]  &  &\hspace{0.5cm} 1 \end{psmallmatrix}.
\]
The top-right entry in this matrix product equals
\[
q(a+k_1,b+k_2)-n_1a(n_2(a+k_1)+n_3(b+k_2))-q(a,b)+n_1a(n_2a+n_3b).
\]
Simplifying this, we obtain the following 4-variable polynomial:
\[
q'(a,b,k_1,k_2)=q(a+k_1,b+k_2)-q(a,b)-n_1a(n_2k_1+n_3k_2).
\]
Now the assumed $(\Z^2,\Gamma)$-consistency implies that for every integer values of $k_1,k_2$ we have $q'(a,b,k_1,k_2)\in \Z$ for every $a,b\in \R$. We claim that this implies that $q'(a,b,k_1,k_2)$ is in fact independent of $a$ and $b$.\\
\indent To see this, let us split $q'$ into the sum $q_0+q_1$, where $q_0(a,b,k_1,k_2)$ consists of all the monomials of $q'$ that involve at least one of $a$ or $b$, and $q_1(k_1,k_2)$ consists of the monomials of $q'$ involving only $k_1$ or $k_2$. If $q_0$ were not the zero polynomial, then there would have to be integers $m_1,m_2$ such that $(a,b)\mapsto q_0(a,b,m_1,m_2)$ is not the zero polynomial. However, our assumption then implies that for every $a,b\in \R$ we have $q_1(m_1,m_2)+q_0(a,b,m_1,m_2)\in \Z$, which is impossible if $(a,b)\mapsto q_0(a,b,m_1,m_2)$ is not identically zero. Hence $q_0(a,b,k_1,k_2)$ must be identically zero and our claim follows.\\
\indent We thus have a 2-variable real polynomial $q'(k_1,k_2)$ that satisfies
\begin{equation}\label{eq:qid}
q'(k_1,k_2)=q(a+k_1,b+k_2)-q(a,b)-n_1a(n_2k_1+n_3k_2),\;\;\forall a,b,k_1,k_2\in \R.
\end{equation}
Let us now consider the partial derivatives of $q'$ around $(0,0)$.\\ 
\indent Noting that $q'(0,0)=0$ (by \eqref{eq:qid}), we have for every $a,b\in \R$ that
\[
\partial_1 q'(0,0)  := \lim_{\epsilon\to 0} \frac{q'(\epsilon,0)}{\epsilon}
 =  \lim_{\epsilon\to 0} \frac{q(a+\epsilon,b)-q(a,b)}{\epsilon}-n_1n_2 a = \partial_1 q(a,b) -n_1n_2 a.
\]
Letting $c_1$ denote the constant $\partial_1 q'(0,0)$, we deduce that 
$\partial_1 q(a,b)= n_1n_2 a +c_1$ for every $a,b\in \R$. Integrating with respect to $a$, this implies that for some 1-variable real polynomial $t$ we have 
\begin{equation}\label{eq:1stvar}
q(a,b)=n_1n_2 a^2/2 +c_1a+t(b).
\end{equation}
With respect to the second variable, we have for every $a,b\in \R$ that
\[
\partial_2 q'(0,0):= \lim_{\epsilon\to 0} \frac{q'(0,\epsilon)}{\epsilon}
 =  \lim_{\epsilon\to 0} \frac{q(a,b+\epsilon)-q(a,b)}{\epsilon}-n_1n_3 a = \partial_2 q(a,b)-n_1n_3 a.
\]
Denoting the constant $\partial_2 q'(0,0)$ by $c_2$, we have then 
$\partial_2 q(a,b) = n_1n_3 a +c_2$ for all $a,b\in \R$. Hence for some 1-variable real polynomial $t'$ we have
\begin{equation}\label{eq:2ndvar}
q(a,b) = n_1n_3 a b +c_2 b+ t'(a).
\end{equation}
The equations \eqref{eq:1stvar} and \eqref{eq:2ndvar} are consistent only if $n_1n_3 ab$ is identically zero, that is only if $n_1n_3=0$.
\end{proof}
\noindent Fixing $\epsilon>0$ sufficiently small, we have that Lemma \ref{lem:calc} contradicts Lemma \ref{lem:eps}, and this completes the proof of Proposition \ref{prop:countercomp3}.

\section{Final remarks}\label{sec:remarks}

\noindent In the last decade, a general approach has emerged in the study of very large combinatorial structures which proceeds by relating these structures to infinite continuous objects. A central example is the notion of limit objects for  convergent sequences of graphs, objects which can be represented by symmetric two-variable measurable functions $W: [0, 1]^2\to [0, 1]$ called \emph{graphons} \cite{L&S}. Analogous limit objects can be defined in the arithmetic setting  for certain notions of convergence for sequences of functions on abelian groups \cite[\S 6.2]{Szegedy:Limits}. The concept in this paper of a continuous model for the convergent sequences $m_{\Lf}(\alpha,\Zmod{p})$ can be viewed as part of the general approach, but it differs from the study of limit objects in a significant way that we would like to emphasize here. \\
\indent A useful definition of convergence for a sequence of functions $(f_n)$ on abelian groups $Z_n$ should guarantee in particular that certain averages of $f_n$ with respect to linear configurations also converge. This is the case for instance in  \cite{Szegedy:Limits} for linear configurations of complexity 1. One of the central purposes of a limit object for such a sequence $(f_n)$, in this case a measurable function $f$ on some compact abelian group, is that the average of $f$ for each such configuration should be \emph{equal} to the limit of the same averages for $f_n$ over $Z_n$. In this sense the notion of limit object involves \emph{exactness}. By contrast, the notion of a continuous model $X$ requires only that one be able to \emph{approximate} the average of $f_n$ by the average of a function over $X$, with arbitrary prescribed accuracy, provided that $n$ is sufficiently large. With this notion of a model, we lose the exactness, but we gain in that the model can be much simpler than a limit object. An example of this gain is the fact that in \cite[Theorem 1]{Szegedy:Limits} the limit object for systems of complexity 1 on groups $\Zmod{p}$ is a compact abelian group with a priori unbounded dimension, whereas in \cite{CS1} it is shown that for modelling such systems the circle group suffices.\\
\indent Let us mention that there is an alternative way to phrase most results in this paper, namely by working with nilspaces and morphisms between them. While this language can simplify certain formulations and be clearer conceptually, there is presently less background literature on it than on filtered nilmanifolds and polynomial maps, so we leave the use of this language for future work.\\
\indent Finally, note that the notion of modelling from Definition \ref{def:modelling} yields a preorder on the set of filtered nilmanifolds, namely the preorder defined by $\ns \leq \nss$ if and only if $\ns$ models $\nss$. Solving the following problem concerning this preorder would yield a generalization of Theorem \ref{thm:main}.
\begin{problem}
For each $s>2$, find a filtered nilmanifold $\ns_s$ of minimal dimension that is a lower bound, in the modelling preorder, for the set of filtered nilmanifolds of  degree at most $s$.
\end{problem}

\begin{appendix}
\section{Background notions and results}\label{App}

\noindent By a filtration $G_\bullet$ on a group $G$ we mean a sequence of nested subgroups $G_{(0)} \rhd G_{(1)} \rhd \ldots$ of $G$ with $G_{(0)}=G_{(1)}=G$ and such that for every $i,j \geq 0$ the group of commutators $[G_{(i)},G_{(j)}]$ is contained in $G_{(i+j)}$. We say that $G_\bullet$ has degree at most $d$ if $G_{(d+1)}=\{\id_G\}$.

We shall use the ring structure on $\R^t$. With this structure, we can define binomial coefficients of vectors $v=\begin{psmallmatrix}v(1)\\[-0.4em]\vdots \\[0.1em] v(t)\end{psmallmatrix}\in \R^t$, writing $\binom{v}{k}$ for the vector $\begin{psmallmatrix}\binom{v(1)}{k}\\[-0.4em]\vdots \\[0.1em] \binom{v(t)}{k}\end{psmallmatrix}\in \R^t$, for $k$ a non-negative integer.

\begin{defn}[The Leibman group]\label{def:LeibGp}
  Let $\Lf = (\lf_1,\ldots,\lf_t)$ be a collection of linear forms $\lf_1,\ldots,\lf_t: \Z^D \to \Z$, let us view $\Lf$ as a matrix in $\Z^{t\times D}$, and let $v_1,\dots,v_D$ denote the columns of $\Lf$. For each $i \geq 1$,  we denote by $\Lf^{[i]}$ the subgroup of $\Z^t$ generated by the elements $\prod_{j\in [D]}\binom{v_j}{k_j}$ where $1\leq \sum_{j\in [D]} k_j\leq i$. Given a filtered group $(G, G_\bullet)$, the \emph{Leibman group} for $\Lf$ on $(G,G_\bullet)$ is the following normal subgroup of $G^t$:
\[
(G,G_\bullet)^{\Lf} = \left\langle g^v : g\in G_{(i)},\; v \in {\Lf}^{[i]},\;i\geq 1\right\rangle,
\]
where $g^v:=\begin{psmallmatrix}g^{v(1)}\\[-0.4em]\vdots \\[0.1em] g^{v(t)}\end{psmallmatrix}\in G^t$.
\end{defn}
\noindent This is essentially Leibman's original definition \cite[\S 5.6]{Leibdiag}, except that the filtration here is a general one (instead of the lower central series) and we are not assuming that the subgroup generated by the columns of $\Lf$ contains $\begin{psmallmatrix} 1\\[-0.4em]\vdots \\[0.1em] 1\end{psmallmatrix}$ (unlike in \cite[\S 5.2]{Leibdiag}).\footnote{In \cite[Definition 3,1]{GTarith} the definition of the Leibman group involves real subspaces of $\R^t$ rather than subgroups of $\Z^t$, and is restricted to connected and simply-connected Lie groups $G$. However, for such groups, that definition agrees with the one given here; this fact was already noted in \cite[\S 5.9]{Leibdiag}.}\\

\noindent Recall the following general notion of polynomial maps between filtered groups (this is a special case of \cite[Definition B.1]{GTZ}).
\begin{defn}[Polynomial maps]\label{def:polymaps}
Let $(G,G_\bullet)$ and $(H,H_\bullet)$ be filtered groups. Given a map $g:H\to G$, and $h\in H$, we define the map $\partial_h g:H\to G$ by $\partial_h g(x) = g(x)^{-1} g(xh)$. 
We say that $g$ is a \emph{polynomial map} adapted to $H_\bullet,G_\bullet$ if for every non-negative integers $i_1,i_2,\ldots,i_n$ and elements $h_j\in H_{(i_j)}$, $j\in[n]$, we have $\partial_{h_1} \partial_{h_2}\cdots \partial_{h_n} g (x) \in G_{(i_1+\cdots + i_n)}$ for all $x\in H$. The set of such maps is denoted $\poly(H_\bullet,G_\bullet)$. We denote by $\poly_0(H_\bullet,G_\bullet)$ the subset of $\poly(H_\bullet,G_\bullet)$ consisting of those maps $g$ satisfying $g(\id_H)=\id_G$.
\end{defn}
\noindent We recall two central facts about these maps: firstly, that $\poly(H_\bullet,G_\bullet)$ with pointwise product is a group (this is the Lazard-Leibman theorem \cite{Laz,Leib}); secondly, that polynomial maps enjoy the composition property, and so filtered groups with polynomial maps between them form a category (see \cite[Corollaries B.4 and B.5]{GTZ}). Note that the same properties hold for $\poly_0(H_\bullet,G_\bullet)$.

Note also that if $G,H$ are abelian and we let $H_\bullet$ be the lower central series and $G_\bullet$ be the maximal degree-$d$ filtration (cf. the paragraph before Theorem \ref{thm:main}), then $g\in \poly(H_\bullet,G_\bullet)$ if and only if $g$ is polynomial of degree $\leq d$ in the sense of \eqref{eq:abpolymap}.

In Definition \ref{def:multib} and elsewhere we use the fact that polynomial maps with trivial constant term preserve Leibman groups. We justify this as follows.

\begin{proposition}\label{prop:polykeepLeib}
Let $(G,G_\bullet),(G',G_\bullet')$ be filtered groups, let $\eqm\in \poly_0(G_\bullet,G_\bullet')$, and let $\Lf:\Z^D\to\Z^t$ be a system of linear forms. Then $g^t$ maps $(G,G_\bullet)^\Lf$ into $(G',G'_\bullet)^\Lf$.
\end{proposition}

\begin{proof}
Let us view $\Lf$ as a matrix in $\Z^{t\times D}$ as in Definition \ref{def:LeibGp}, with its $i$-th row identified with $\lf_i$ as an element of $\Z^D$. Now note the following alternative characterization of the Leibman group:
\begin{equation}\label{eq:LeibGpChar}
(G,G_\bullet)^\Lf = \left\{\begin{psmallmatrix}h(\lf_1)\\[-0.4em]\vdots \\[0.1em] h(\lf_t)\end{psmallmatrix}: h\in \poly_0\big(\Z^D,G_\bullet\big)\right\}.
\end{equation}
This was already observed in \cite[\S 5.7]{Leibdiag} in the case of the lower-central series $G_\bullet=G_{\bullet(1)}$. Let us outline a proof for general filtrations. To show that the right side of \eqref{eq:LeibGpChar} is included in $(G,G_\bullet)^\Lf$, one can use the fact that every $h\in \poly_0(\Z^D,G_\bullet)$ has a Taylor expansion (see \cite[Lemma A.1]{GTarith}) with trivial $0$-th coefficient. Substituting this expansion into $\begin{psmallmatrix}h(\lf_1)\\[-0.4em]\vdots \\[0.1em] h(\lf_t)\end{psmallmatrix}$, we find that this is a product of elements of the form $g^{\prod_{j\in [D]}\binom{v_j}{k_j}}$ where $1\leq \sum_{j\in [D]} k_j\leq i$ and $g\in G_{(i)}$, whence these elements are all in $(G,G_\bullet)^\Lf$, and the inclusion follows. The opposite inclusion follows from the fact that for each $g\in G_{(i)}$ and each $v$  generating $\Lf^{[i]}$ (thus $v=\prod_{j\in [D]}\binom{\Lf(e_j)}{k_j}$ for the elements $e_j$ of the standard basis of $\R^D$), the element $g^v=\begin{psmallmatrix}g^{v(1)}\\[-0.4em]\vdots \\[0.1em] g^{v(t)}\end{psmallmatrix}$ is of the form $\begin{psmallmatrix}h(\lf_1)\\[-0.4em]\vdots \\[0.1em] h(\lf_t)\end{psmallmatrix}$ where $h$ is the map $(n_1,\dots,n_D)\mapsto g^{\prod_{j\in [D]}\binom{n_j}{k_j}}$. By the same Taylor expansion result as above, the map $h$ is in $\poly_0(\Z^D,G_\bullet)$, so $g^v$ is in the right side of \eqref{eq:LeibGpChar}. The inclusion is then fully deduced using the group property of $\poly_0(\Z^D,G_\bullet)$.\\
\indent To prove the main claim in the proposition, note that, by the composition property recalled above, we have for every $h\in  \poly_0(\Z^D,G_\bullet)$ that $g\circ h \in  \poly_0(\Z^D,G_\bullet')$. Hence for every  element $\begin{psmallmatrix}h(\lf_1)\\[-0.4em]\vdots \\[0.1em] h(\lf_t)\end{psmallmatrix}\in (G,G_\bullet)^\Lf$ we have $g^t\begin{psmallmatrix}h(\lf_1)\\[-0.4em]\vdots \\[0.1em] h(\lf_t)\end{psmallmatrix}=\begin{psmallmatrix}g\circ h(\lf_1)\\[-0.4em]\vdots \\[0.1em] g\circ h(\lf_t)\end{psmallmatrix}\in (G',G_\bullet')^\Lf$. 
\end{proof}

\noindent We now discuss nilmanifolds.

\begin{defn}[Filtered nilmanifold]\label{def:Filnil}
A \emph{filtered nilmanifold} $(G/\Gamma,G_\bullet)$ of degree at most $d$ consists of a connected and simply-connected Lie group $G$, a lattice $\Gamma$ in $G$ (i.e. a discrete cocompact subgroup) and a filtration $G_\bullet$ of degree at most $d$ on $G$, with each group $G_{(i)}$ being a closed and connected subgroup of $G$ such that $\Gamma_{(i)}:=\Gamma\cap G_{(i)}$ is a lattice in $G_{(i)}$. If a Mal'cev basis\footnote{See \cite[Definitions 2.1 and 2.4]{GTOrb} for the definition of a Mal'cev basis and for the notion of quantitative rationality for such bases.} $\X$ on $G/\Gamma$ adapted to $G_\bullet$ has been fixed, we refer to $(G/\Gamma,G_\bullet,\X)$ as a \emph{filtered based nilmanifold}. The \emph{complexity} of such a nilmanifold is the least $C>0$ such that the topological dimension of $G$, the degree of $G_\bullet$, and the rationality of $\X$ are all at most $C$.
\end{defn}

\begin{remark}\label{rem:disconnected}
In this paper we occasionally use the term \emph{possibly disconnected filtered nilmanifold} to refer to a couple $(G/\Gamma,G_\bullet)$ where  $G$ is a simply-connected but not necessarily connected Lie group, where $G_\bullet$ is a filtration of closed (simply-connected) subgroups $G_{(i)}$ of $G$, and where $\Gamma$ is a lattice in $G$ such that $\Gamma\cap G_{(i)}$ is a lattice in $G_{(i)}$. Let us emphasize that, by the term ``filtered nilmanifold" without the additional term ``possibly disconnected", we always refer to the notion in Definition \ref{def:Filnil}.
\end{remark}

\noindent When the filtration $G_\bullet$ on $G$ is clear from the context, we abbreviate the notation $(G,G_\bullet)^\Lf$ to $G^\Lf$.

\begin{defn}[Leibman nilmanifold]\label{def:Leibnil}
Let $\ns=(G/\Gamma,G_\bullet)$ be a filtered nilmanifold of finite degree, and let $\Lf:\Z^D\to \Z^t$ be a system of linear forms. The \emph{Leibman nilmanifold} for $\Lf$ on $\ns$, denoted $\ns^\Lf$, is the  nilmanifold $G^\Lf/\Gamma^\Lf$.
\end{defn}
\noindent For proofs that $G^\Lf$ and $\Gamma^\Lf$ have the required properties for $G^\Lf/\Gamma^\Lf$ to be a nilmanifold, we refer the reader to \cite[\S 3]{GTarith}.\\
\indent We now recall the definition of complexity of a system of forms, which was introduced by Gowers and Wolf (who called it ``true complexity" to distinguish it from an earlier complexity notion introduced by Green and Tao) and studied in a series of papers beginning with \cite{GWcomp}. 

\begin{defn}[Size and complexity of systems of linear forms]\label{def:size&comp}
We say that a system $\Lf = (\lf_1, \ldots, \lf_t)$ of linear forms $\lf_i : \Z^D \to \Z$ has \emph{size at most $L$} if $D,t \leq L$ and the coefficients of each $\lf_i$ have absolute value at most $L$. We say that $\Lf$ has  \emph{complexity} $s$ (on simple abelian groups) if for every $\eta>0$ there exists $\lambda=\lambda(\Lf,\eta)>0$ such that, for every prime $p$, for every $f,f':\Zmod{p}\to\C$ bounded by 1 and satisfying $\|f-f'\|_{U^{s+1}}\leq \lambda$, we have $|\Sol_\Lf(f:\Zmod{p})-\Sol_\Lf(f':\Zmod{p})|\leq \eta$, and $s$ is the least integer with this property.
\end{defn}
\noindent There are equivalent formulations of this notion. The main one of these was originally conjectured by Gowers and Wolf in \cite{GWcomp}, and states that $\Lf$ has complexity $s$ if the powers $\lf_i^{s+1}$ are linearly independent and $s$ is the least integer with this property (here we view the forms $\lf_i$ as linear polynomials with integer coefficients in $\R[x_1,\dots,x_D]$); see \cite{GW,GWcomp}, and also Theorem 7.1 and the remark at the end of Section 7 in \cite{GTarith}. This formulation makes it clear that every form in a system of finite complexity must be non-zero. This fact is relevant, for instance, in relation to  the $L^1$-continuity mentioned in Remark \ref{rem:L1cont}, as it implies the surjectivity of the coordinate projections mentioned in that remark.

Let us now turn to Theorem \ref{thm:Zp-to-X}, which we restate here.
\begin{theorem}\label{thm:Zp-to-X-app}
For every $\delta >0$ there exists $C>0$ such that the following holds. For every prime $p\geq C$ and every function $f : \Zmod{p} \to [0,1]$, there is a filtered based nilmanifold $\ns=(G/\Gamma, G_\bullet, \X)$ of degree at most $2$ and complexity at most $C$, and a continuous function $F : G/\Gamma \to [0,1]$ satisfying $\norm{F}_{\Lip(\X)} \leq C$, such that for every system $\Lf$ of integer linear forms of size at most $1/\delta$ and complexity at most 2, we have
\begin{equation}\label{eq:Zp-to-X-app}
\abs{\, \Sol_\Lf(f : \Zmod{p}) - \Sol_\Lf(F: \ns )\,} \leq \delta.
\end{equation}
\end{theorem}
\noindent The proof is essentially a combination of \cite[Theorem 5.1]{CS3} (a regularity result for the $U^d(\Zmod{p})$ norm) with \cite[Theorem 4.1]{CS3} (a counting result for balanced maps on $\Zmod{p}$), but with a couple of additional observations. 
\begin{proof}
Let $\eta > 0$ and let $\mathcal{F} : \R^+ \to \R^+$ be a growth function, both to be specified in terms of $\delta$ later, and let $\lambda$ be as in Definition \ref{def:size&comp}. If $p\gg_{\eta,\mathcal{F}} 1$, we may apply the regularity result \cite[Theorem 5.1]{CS3} for the $U^3$ norm to $f$, to obtain a decomposition $f = f_\nil + f_\sml + f_\unf$ and an integer $Q = O_{\eta,\delta,\mathcal{F}}(1)$ with the following properties:
\begin{enumerate}
\item There is a nilmanifold $\ns=(G/\Gamma, G_\bullet, \X)$ of degree at most $2$ and complexity at most $Q$, an $\mathcal{F}(Q)$-irrational sequence $g \in \poly(\Z,G_\bullet)$ that is\footnote{In the terminology of this paper, $g$ is a $(p\Z,\Gamma)$-consistent sequence in $\poly(\Z,G_\bullet)$ such that $g(n)\Gamma$ is $c$-balanced with $c=o(1)_{\mathcal{F}(Q)\to\infty}$ . For the definition of irrationality, see \cite[Definition 4.7]{CS3}.} $p$-periodic mod $\Gamma$, and $F : G/\Gamma \to \C$ with $\norm{F}_{\Lip(\X)} \leq Q$, such that $f_\nil = F(g(n)\Gamma)$.  Note that we also have $g(0)=\id_G$, indeed this follows from the factorization result \cite[Proposition 5.2]{CS3} as used in the proof of \cite[Theorem 5.1]{CS3}.
\item $\norm{f_\sml}_2 \leq \eta$,
\item $\norm{f_\unf}_{U^3} \leq 1/\mathcal{F}(Q)$, and
\item $f_\nil$ and $f_\nil + f_\sml$ take values in $[0,1]$.
\end{enumerate}
Furthermore, since $f_\nil$ is $[0,1]$-valued, we may assume that $F$ is real-valued by taking real parts, and then by replacing it with $\max(\min(F,1),0)$ we may in fact assume that it is also $[0,1]$-valued; neither of these operations can increase $\norm{F}_{\Lip(\X)}$.

Now let $\Lf = (\lf_1, \ldots, \lf_t)$ be any system of linear forms $\lf_i : \Z^D \to \Z$ of size at most $1/\delta$ and complexity at most 2. Then by the choice of $\lambda$ and a standard bound for $\Sol_\Lf(f)$ in terms of $\|f\|_2$ (see \cite[inequality (1)]{CS3}), we have
\begin{equation}\label{eq:fnilreduc}
 \abs{ \Sol_\Lf(f:\Zmod{p}) - \Sol_\Lf(f_\nil:\Zmod{p}) } =o_\delta(1)\;\;\textrm{ as }\eta\to 0\textrm{ and }\mathcal{F}(Q)\to\infty.
\end{equation}
We now deal with $\Sol_\Lf(f_\nil: \Zmod{p})$ using the counting result  \cite[Theorem 4.1]{CS3}. The Lipschitz function that we use is $F^{\otimes t} : G^t/\Gamma^t \to \C$, $(x_1, \ldots, x_t) \mapsto F(x_1)\cdots F(x_t)$. This satisfies $\norm{F^{\otimes t}}_{\Lip(\X^t)} = O_{Q,\delta}(1)$ by \cite[Lemma A.4]{CS3}. Applying \cite[Theorem 4.1]{CS3} to this function, with parameter $M=O_{\delta,\eta,\mathcal{F}}(1)$, we obtain
\begin{equation}\label{eq:countingappli}
\Sol_\Lf(f_\nil:\Zmod{p}) = \E_{n\in \Zmod{p}^D} F^{\otimes t}(g^t(\Lf(n))\Gamma^t) = \int_{G^\Lf/\Gamma^\Lf} F^{\otimes t} + o_M(1)_{\mathcal{F}(Q) \to \infty}. 
\end{equation}
Combining \eqref{eq:fnilreduc} and \eqref{eq:countingappli}, we now set $\eta=\eta(\delta)$ sufficiently small, and $\mathcal{F}$ with sufficiently fast growth in terms of $\delta$, to obtain \eqref{eq:Zp-to-X}, noting that $ \int_{G^\Lf/\Gamma^\Lf} F^{\otimes t}=\Sol_\Lf(F: \ns )$.
\end{proof}
\end{appendix}

\noindent \textbf{Acknowledgements.} This work was supported by the ERC Consolidator Grant No. 617747. The authors also thank an anonymous referee for very useful remarks.


\begin{thebibliography}{1}
\bibitem{CS1} P. Candela, O. Sisask, \emph{On the asymptotic maximal density of a set avoiding solutions to linear equations modulo a prime}, Acta Math. Hungar. \textbf{132} (2011) (3), 223--243.
\bibitem{CS3} P. Candela, O. Sisask, \emph{Convergence results for systems of linear forms on cyclic groups, and periodic nilsequences}, SIAM J. Discrete Math. \textbf{28} (2014) (2), 786--810.
\bibitem{croot:3APminconv} E. Croot, \emph{The minimal number of three-term arithmetic progressions modulo a prime converges to a limit}, Canad. Math. Bull. \textbf{51} (2008), no. 1, 47--56.
\bibitem{C&Lev} E. Croot, V. Lev, \emph{Open problems in additive combinatorics}, Additive Combinatorics, CRM Proceedings \& Lecture Notes, vol. 43, 2007.
\bibitem{C&G} L. Corwin, F.P. Greenleaf, \emph{Representations of nilpotent Lie groups and their applications, Part 1: Basic theory and examples}, Cambridge studies in advanced mathematics 18, C.U.P. 1990.
\bibitem{principlesHA} A. Deitmar, S. Echterhoff, \emph{Principles of harmonic analysis} (Springer, 2009).
\bibitem{Eisner-Tao} T. Eisner, T. Tao, \emph{Large values of the Gowers-Host-Kra seminorms}, J. Anal. Math. \textbf{117} (2012), 133--186. 
\bibitem{GSz} W.T. Gowers, \emph{A new proof of Szemer\'edi's theorem}, GAFA \textbf{11} (2001), 465--588.
\bibitem{GW} W.T. Gowers, J. Wolf, \emph{Linear functions and quadratic uniformity for functions on $\Z_N$}, J. Anal. Math. \textbf{115} (2011),  121--186.
\bibitem{GWcomp} W.T. Gowers, J. Wolf, \emph{The true complexity of a system of linear equations}, Proc. Lond. Math. Soc. (3) \textbf{100} (2010), no. 1, 155--176. 
\bibitem{GTMob} B. Green, T. Tao, \emph{Quadratic uniformity of the M\"obius function}, Ann. Inst. Fourier (Grenoble) \textbf{58} (2008), no. 6, 1863--1935. 
\bibitem{GTOrb} B. Green, T. Tao, \emph{The quantitative behaviour of polynomial orbits on nilmanifolds},  Ann. of Math. (2) \textbf{175} (2012), no. 2, 465--540. 
\bibitem{GTarith} B. Green, T. Tao, \emph{An arithmetic regularity lemma, an associated counting lemma, and applications}, An irregular mind, 261-334, Bolyai Soc. Math. Stud., 21, Janos Bolyai Math. Soc., Budapest, 2010.
\bibitem{GTZ} B. Green, T. Tao, T. Ziegler, \emph{An inverse theorem for the Gowers $U^{s+1}[N]$-norm}, Ann. of Math. \textbf{176} (2012), no. 2, 1231--1372. 
\bibitem{Laz} M. Lazard, \emph{Sur les groupes nilpotents et les anneaux de Lie}, Ann. Sci. Ecole Normale Sup. (3) \textbf{71} (1954), 101-190.
\bibitem{Leib} A. Leibman, \emph{Polynomial sequences in groups}, Journal of Algebra \textbf{201} (1998), 189-206.
\bibitem{Leibdiag} A. Leibman, \emph{Orbit of the diagonal in the power of a nilmanifold}, Trans. Amer. Math. Soc. \textbf{362} (2010), no. 3, 1619--1658.
\bibitem{L&S} L. Lov\'asz, B. Szegedy, \emph{Limits of dense graph sequences}, J. Comb. Theory, Ser. B \textbf{96} (2006), 933--957.
\bibitem{O&V} A. L. Onishchik, E. B. Vinberg, \emph{Foundations of Lie Theory and Lie transformation groups}, Springer 1997.
\bibitem{SisaskPhD} O. Sisask, \emph{Combinatorial properties of large subsets of abelian groups}, Ph.D. Thesis, University of Bristol (2009).
\bibitem{Szegedy:Limits} B. Szegedy, \emph{Limits of functions on groups}, preprint. \href{http://arxiv.org/abs/1502.07861}{	arXiv:1502.07861}.
\bibitem{Szegedy:HFA} B. Szegedy, \emph{On higher order Fourier analysis}, preprint. \href{http://arxiv.org/abs/1203.2260}{arXiv:1203.2260}.
\bibitem{T-V} T. Tao, V. Vu, \emph{Additive combinatorics}, Cambridge University Press, 2006.
\end{thebibliography}
\end{document}